\documentclass[12pt]{article}

\usepackage{times}

\usepackage[utf8]{inputenc}
\usepackage{commath}
\usepackage{amsthm}
\usepackage{amsmath}
\usepackage{amssymb}
\usepackage{cite}
\usepackage{setspace}

\usepackage{bigints}
\usepackage{comment}
\usepackage{mathtools}
\usepackage{mathrsfs}
\usepackage{fancyhdr}

\allowdisplaybreaks[4]

\numberwithin{equation}{section}
\usepackage{etoolbox}
\patchcmd{\thebibliography}
  {\settowidth}
  {\setlength{\itemsep}{0pt plus -10pt}\settowidth}
  {}{}
\apptocmd{\thebibliography}
  {
	\footnotesize  
  }
  {}{}

\makeatletter
\newtheorem*{rep@theorem}{\rep@title}
\newcommand{\newreptheorem}[2]{%
\newenvironment{rep#1}[1]{%
 \def\rep@title{#2 \ref{##1}}%
 \begin{rep@theorem}}%
 {\end{rep@theorem}}}
\makeatother

\newreptheorem{theorem}{Theorem}

\newtheorem{thm}{Theorem}[section]
\newtheorem{prop}[thm]{Proposition}
\newtheorem*{prop*}{Proposition}
\newtheorem{defn}[thm]{Definition}
\newtheorem{rem}[thm]{Remark}
\newtheorem*{thm*}{Theorem}
\newtheorem{cor}[thm]{Corollary}
\newtheorem*{cor*}{Corollary}
\newtheorem{lem}[thm]{Lemma}

\newtheorem{notat}[thm]{Notation}

\title{On the full asymptotics of analytic torsion} 

\author
{Siarhei Finski}

\date{}

\usepackage[%
    left=1in,%
    right=1in,%
    top=1.1in,%
    bottom=0.8in,%
    paperheight=11in,%
    paperwidth=8.5in%
]{geometry}

\newcommand{\imun} {\sqrt{-1}}

\newcommand{\comp}{\mathbb{C}}
\newcommand{\real}{\mathbb{R}}
\newcommand{\rat}{\mathbb{Q}}
\newcommand{\nat}{\mathbb{N}}
\newcommand{\integ}{\mathbb{Z}}

\newcommand{\tinyint}{\begingroup\textstyle\int\endgroup}

\newcommand{\enmr}[1]{{\rm{End}}(#1)}

\newcommand{\dfor}[2][]{\Omega^{#1}(#2)}
\newcommand{\cinf}[1]{\mathscr{C}^{\infty} ( #1 )}

\newcommand{\dbar}{ \overline{\partial} }
\newcommand{\laplcomp}{\Box}

\newcommand{\rk}[1]{{\rm{rk}} ( #1 )}
\newcommand{\tr}[1]{{\rm{Tr}} \big[ #1 \big]}
\newcommand{\str}[1]{{\rm{Tr}_s} \big[ #1 \big]}
\newcommand{\matcirc}[1]{\mathring{#1}}

\renewcommand{\Re}{\operatorname{Re}}
\renewcommand{\Im}{\operatorname{Im}}
\newcommand{\scal}[2]{\langle #1, #2 \rangle}
\newcommand{\mell}[1]{{\rm{M}} \big[ #1 \big]}
\newcommand{\spec}{{\rm{Spec}}}

\newcommand{\vast}{\bBigg@{4}}
\newcommand{\Vast}{\bBigg@{5}}

\newcommand{\mm}{\mathcal{M}}
\newcommand{\ee}{\mathcal{E}}
\newcommand{\lin}{\mathcal{L}}
\newcommand{\nn}{\mathcal{N}}

\newcommand{\td}{{\rm{Td}}}
\newcommand{\ch}{{\rm{ch}}}

\newcommand{\Addresses}{{
  \bigskip
  \footnotesize
  \noindent \textsc{Siarhei Finski, UFR de Mathématiques, IMJ-PRG, Université Paris Diderot-Paris 7, Case 7012, 8 Place Aurélie Nemours 75205 Paris Cedex 13, France.}\par\nopagebreak
  \noindent  \textit{E-mail }: \texttt{siarhei.finski@imj-prg.fr}.
}}

\newenvironment{sciabstract}{}

\begin{document} 
\clearpage\maketitle
\thispagestyle{empty}

\begin{sciabstract}
  \textbf{Abstract.} The purpose of this article is to study the asymptotic expansion of Ray-Singer analytic torsion associated with powers $p$ of a given positive line bundle over a compact $n$-dimensional complex manifold, as $p \to \infty$. Here we prove that the asymptotic expansion contains only the terms of the form $p^{n-i} \log p, p^{n-i}$ for $i \in \nat$. For the first two leading terms it was proved by Bismut-Vasserot. We calculate the coefficients of the terms $p^{n - 1} \log p, p^{n - 1}$ in the Kähler case and thus answer the question posed in the recent work of Klevtsov-Ma-Marinescu-Wiegmann about quantum Hall effect. Our second result concerns the general asymptotic expansion of the analytic torsion for a compact complex orbifold.
\end{sciabstract}

\tableofcontents

\section{Introduction}
	The holomorphic analytic torsion was introduced by Ray-Singer in \cite{Ray73}. It is a number $T(g^{TM},h^E)$ defined for a holomorphic Hermitian vector bundle $(E, h^E)$ over a compact Hermitian manifold $(M, g^{TM}, \Theta)$ as the regularized determinant of the Kodaira Laplacian $\laplcomp^E = \dbar{}^{E} \, \dbar{}^{E*} + \dbar{}^{E*} \dbar{}^{E}$, acting on the vector space of sections of the vector bundle $\Lambda^{\bullet} (T^{*(0,1)}M) \otimes E$. 
	\begin{sloppypar}
	Let $L$ be a positive Hermitian line bundle over $M$, $\dim_{\comp} M = n$. In \cite{BVas}, Bismut-Vasserot obtained the asymptotics of $\log T(g^{TM}, h^{L^p \otimes E})$ as $p \to + \infty$, (here $L^p := L^{\otimes p}$), and they gave an explicit formula for the coefficients of the leading terms $p^n \log p, p^n$ of the expansion. This asymptotic expansion played an important role in a result of arithmetic ampleness  (see Gillet-Soulé \cite{GilSoul92}, \cite[Chapter VIII]{Soule92}). In this article we obtain a general formula for it in the orbifold's setting. The general strategy of the proof is the same as in the article \cite{BVas}: we study this asymptotic expansion by studying the heat kernel of the rescaled Kodaira Laplacian $\laplcomp^{L^p \otimes E} / p$. We use functional analysis approach inspired by Bismut-Lebeau\cite{BisLeb91} and realized in Ma-Marinescu \cite[\S 5.5]{MaHol}. Certainly, one expects that the probability approach of \cite{BVas} could also be applied.
	In Theorem \ref{dzeta_asympt} we also give an explicit formula for the coefficients of the subsequent terms $p^{n-1} \log p, p^{n-1}$.
	\end{sloppypar}
	\begin{sloppypar}
		Now let's describe our results more precisely. Let $(M, g^{TM}, \Theta)$ be a compact Hermitian manifold of complex dimension $n$. Let $(E, h^E)$ be a holomorphic Hermitian vector bundle over $M$ with first Chern class $ c_1(E)$ and rank $\rk{E}$.
	Let $(L, h^L)$ be a Hermitian positive line bundle over $M$.
	Let's denote by $\omega$ the 2-form defined by
	\begin{equation}\label{defn_omega}
		\omega := c_1(L,h^L) := \frac{\imun}{2 \pi}  R^L,	
	\end{equation}
	where $R^L$ is the curvature of the Chern connection on $(L, h^L)$. We define $\matcirc{R^L} \in \enmr{T^{(1,0)}M}$ by
	\begin{equation}\label{defn_R_L_0}
		g^{TM}(\matcirc{R^L}U, \overline{V}) = R^L(U, \overline{V}), \quad U,V \in T^{(1,0)}M.
	\end{equation}
	We denote by $T (g^{TM}, h^{L^p \otimes E})$ the analytic torsion of $L^p \otimes E$ associated with $g^{TM}, h^L, h^E$ (see Definition \ref{defn-RS_anal_torsion}). 
	From now on, “a local coefficient" means that it can be expressed as an integral of a density defined locally over $M$.
	Our first result (cf. Theorem \ref{thm-gen_asympt_zet_gen}) is
	\end{sloppypar}
	\begin{thm}\label{thm-gen_asympt_zet}
		There are local coefficients $\alpha_i, \beta_i \in \real, i \in \nat$ such that for any $k \in \nat$, as $p \to + \infty$
		\begin{equation}\label{eqn_gen_asympt_zet111}
			-2 \log T (g^{TM}, h^{L^p \otimes E}) = \textstyle \sum_{i = 0}^{k} p^{n-i}( \alpha_i \log p + \beta_i ) + o(p^{n-k}),
		\end{equation}
		Moreover, the coefficients $\alpha_i$ do not depend on $g^{TM}, h^L, h^E$.
	\end{thm}
	\begin{rem}
		Moreover, in the case if $M$ is the fiber of a proper holomorphic submersion, we prove in Section \ref{subsec_proof_aux} that $ \alpha_i$, $\beta_i$ are smooth over the base of the family, and derivatives  over the base commute with the asymptotics (\ref{eqn_gen_asympt_zet111}).
	\end{rem}
	 We note that in \cite[Theorem 8]{BVas}, Bismut-Vasserot proved Theorem \ref{thm-gen_asympt_zet} for $k=0$. They computed
	\begin{equation}\label{form_B_Vas}
		\begin{aligned}
			 \alpha_0 = \frac{n \, \rk{E}}{2} \int_M \frac{\omega^n}{n!}, \quad
			 \beta_0 = \frac{\rk{E}}{2} \int_M \log \Big( \det \frac{\matcirc{R^L}}{2 \pi} \Big) \frac{\omega^n}{n!}.
		\end{aligned}
	\end{equation}
	\begin{thm}\label{dzeta_asympt}
		If $\Theta = \omega$, we have
		\begin{align}
			& \alpha_1 = \frac{(3n + 1) \rk E}{12} \int_M c_1(TM) \frac{\omega^{n-1}}{(n-1)!} + 
			\frac{n}{2} \int_M c_1(E) \frac{\omega^{n-1}}{(n-1)!}, \label{eqn_alpha_1} \\
			& \beta_1 =  \frac{\rk E}{24} ( 24 \zeta'(-1) + 2 \log(2 \pi) + 7 ) \int_M c_1(TM) \frac{\omega^{n-1}}{(n-1)!} + 
			\frac{1}{2} \int_M c_1(E) \frac{\omega^{n-1}}{(n-1)!}.   \label{eqn_beta_1} 
		\end{align}
	\end{thm}
	\begin{rem}\label{rem_family}
		In the special case when $M$ is a Riemann surface Theorem \ref{dzeta_asympt} gives a precise version of some results concerning quantum Hall effect  in physics, see \cite[p. 839]{KlMa}, \cite[\S 5]{KlPhys}.
	\end{rem}
	In Section \ref{rel_prev_work} we verify this result for the case $M = \comp \mathbb{P}^1, L = \mathcal{O}(1)$ by calculating the coefficients of the asymptotic expansion of $T(g^{FS}, h^{\mathcal{O}(p)})$ as $p \to + \infty$, for the Fubini-Study metric $g^{FS}$. In Section \ref{rel_prev_work} we also discuss the informal relation of our result with the arithmetic Riemann-Roch theorem of Gillet-Soulé \cite{GilSoul92}. We also make a connection with \cite{KlMa}, where Klevtsov-Ma-Marinescu-Wiegmann conjectured \cite[p.839]{KlMa} the coefficient of the term $\log p$ and the constant term for Riemann surfaces. As it turns out, their conjecture is true for $\log p$, but not for the constant term, see Section \ref{rel_prev_work}.
	\par Our last result (cf. Theorem \ref{thm-gen_asympt_orbi_gen} for a precise statement) is a generalization of Theorem \ref{thm-gen_asympt_zet} to the orbifold's case. Let $(\mm, g^{T \mm}, \Theta)$ be a compact effective Hermitian orbifold with strata $\Sigma \mm$ (see Definition \ref{defn_strata}). We denote by $\Sigma \mm^{[j]}$ for $j \in J$ the connected components of $\Sigma \mm$, by $n_j$ it's dimension.
	Let $(\ee, h^{\ee})$ be a proper holomorphic Hermitian orbifold vector bundle (see Definition \ref{defn_orbi_v_b}) on $\mm$ and let $(\lin,h^{\lin})$ be a proper Hermitian positive orbifold line bundle on $\mm$.
	\begin{thm}\label{thm-gen_asympt_orbi}
		There are local coefficients $\widetilde{\alpha_i}, \widetilde{ \beta_i }\in \real$ and $m_j \in \nat, \gamma_{j,i}, \kappa_{j,i} \in \real, j \in J, i \in \nat$ such that we have the following asymptotic expansion for any $k \in \nat$, as $p \to + \infty$
		\begin{multline}\label{eqn_orbi_asympt}
		- 2 \log T (g^{T \mm},h^{\lin^{p} \otimes \ee}) = 
		 \sum_{i = 0}^{k} p^{n-i} \big( \widetilde{\alpha_i} \log p + \widetilde{\beta_i} \big)\\ 
		+  \sum_{i=0}^{k + n_j - n}  \sum_{j \in J} \frac{ p^{n_j - i}}{m_j} e^{\imun \theta_j p} \big( \gamma_{j,i} \log p + \kappa_{j,i}  \big) +  o(p^{n-k}). 
		\end{multline}
		The values $\theta_j, \gamma_{j,i}, \kappa_{j,i}, m_j$ depend only on the local geometry around the singular set of $\mm$, and
		\begin{equation}\label{coroll}
			 \widetilde{\alpha}_0 = \frac{n \, \rk{\ee}}{2} \int_\mm \frac{\omega^n}{n!}, \quad
			 \widetilde{\beta}_0 = \frac{\rk{\ee}}{2} \int_\mm \log \Big( \det \frac{\matcirc{R^{\lin}}}{2 \pi} \Big) \frac{\omega^n}{n!},
		\end{equation}
		where $\tilde{\omega}$ and $\matcirc{R^{\lin}}$ are the orbifold analogues of (\ref{defn_omega}) and (\ref{defn_R_L_0}). There are $c_j \neq 0$ such that
		\begin{equation}\label{eqn_gamma_form}
		\textstyle \gamma_{j, 0} 
		 =
  			\begin{cases} 
      			\hfill c_{j} \int_{\Sigma \mm^{[j]}} \frac{\tilde{\omega}^{n-1}}{(n-1)!},  & \text{ if ${\rm{codim}} \, \Sigma \mm^{[j]} = 1$, } \\
      			\hfill 0,  & \text{ otherwise.} \\
 			\end{cases}
		\end{equation}
		Similarly to the manifold's case, the constants $\widetilde{\alpha}_i$, $\gamma_{j,i}$ do not depend on $g^{T \mm}$ ,$h^{\lin^{p}}$, $h^{\ee}$.
		When $\Theta = \omega$, $\widetilde{\alpha}_1, \widetilde{\beta}_1$ are given by (\ref{eqn_alpha_1}) and (\ref{eqn_beta_1}) after replacing $M$ by $\mm$. 
	\end{thm}
	\begin{cor}
		The set $\{ T (g^{T \mm},h^{\lin^{p} \otimes \ee}) : p \in \nat \}$ detects the singularities of codimension 1.
	\end{cor}
	\begin{rem}
		1. Since for the stabilizers $G_x$ of $x \in \mm$, there are only finitely many possible values of $\{ |G_x|, x \in M \}$, one can take $q \in \nat$, such that $G_x$ acts as identity on $\lin^q_{x}$ for any $x \in M$; thus, $q \theta_j \in 2 \pi \nat$ (see (\ref{defn_theta_j})) and the asymptotic expansion for $p = q k, k \in \nat$ has only terms $p^{n-i} \log p, p^{n-i}$ for $i \in \nat$ in (\ref{eqn_orbi_asympt}).
		\par 2. We see that Theorem \ref{thm-gen_asympt_orbi} is a generalization of Theorem \ref{thm-gen_asympt_zet}, but to facilitate, we present firstly a proof of Theorem \ref{thm-gen_asympt_zet} and then explain the necessary modifications to get Theorem \ref{thm-gen_asympt_orbi}.
		\par 3. The coefficients $\widetilde{\alpha_i}, \widetilde{ \beta_i }$ are the orbifold's versions of $\alpha_i, \beta_i$. Rigorously, this means that each $\alpha_i, \beta_i$ is an integral of a local quantity and $\widetilde{\alpha_i}, \widetilde{ \beta_i }$ are just the integrals of the same quantities defined in an orbifold chart. Thus, by Theorem \ref{thm-gen_asympt_orbi} we see that if the singularities of $\mm$ appear in codimension at least 2, the coefficients of $p^{n-1} \log p, p^{n-1}$ of the expansion of $\log T (g^{T \mm},h^{\lin^{p} \otimes \ee})$ are given by the same formulas as in Theorem \ref{dzeta_asympt}. In general, we may express the coefficient $\kappa_{j,0}$ with the help of Mellin transform (see Theorem \ref{thm-gen_asympt_orbi_gen}, (\ref{form_c_j_0})), but we don't pursue the simplification of this formula.
	\end{rem}
	When an orbifold $\mm$ is obtained as a quotient of a transversal locally free $CR$ $S^1$-action on a smooth CR manifold, Theorem \ref{thm-gen_asympt_orbi} gives a refinement of the main result of Hsiao-Huang \cite{Hsiao16}, see Section \ref{subsect_proof_orbi} for detailed explanation.
	\par Now we describe some history of related problems and propose some directions in which our results might be useful. In the article \cite{BVasSym}, Bismut-Vasserot generalized \cite{BVas} by computing the asymptotic expansion of $\log T(g^{TM}, h^{E \otimes {\rm Sym}^p \zeta })$, as $p \to + \infty$, where $(\zeta, g^{\zeta})$ is a Hermitian Griffiths-positive vector bundle and $(E, h^E)$ is a holomorphic Hermitian vector bundle. Recently, Puchol \cite{Puchol2016} obtained a generalization of this result to the family case. Let's describe his result more precisely. 
	\par Let $\pi : X \to B$ be a proper holomorphic Kähler fibration with a compact fiber $M$ in the sense of \cite[Definition 1.4]{BGS2}, i.e. there exists a closed (1,1)-form $\omega_{\text{fam}}$ such that its restriction on the fibers of $\pi$ gives a Kähler form. Let $(E, h^E)$ be a holomorphic Hermitian vector bundle over $X$. We suppose that the direct image sheaf $R^{\bullet} \pi_* E$ is locally free, i.e. the Dolbeaut cohomology of $E$ along the fibers is a holomorphic bundle.
	In \cite{BisKoh}, Bismut-Köhler introduced the torsion form $\mathcal{T}(\omega_{\text{fam}},h^{E})$, which is a smooth differential form on $B$, satisfying
	\begin{align}
		&  \mathcal{T}(\omega_{\text{fam}},h^{E})^{[0]} = - 2 \log T(g^{TM},h^{E}), \text{ where $[0]$ denotes 0-degree component},\\ \nonumber
		&  \frac{\dbar \partial}{2 \pi \imun} \mathcal{T}(\omega_{\text{fam}},h^{E}) = \sum_i (-1)^i \ch (H^i(M,E|_{M}), h^{H^i(M,E|_{M})}) - \int_M \td (TM, h^{TM}) \ch(E, h^E),
	\end{align}
	where $h^{H^{\bullet}(M,E|_{M})}$ is $L^2$-metric, and $\ch (\cdot, \cdot)$, $\td (\cdot, \cdot)$ are the corresponding Chern and Todd forms. In particular, we see that the second identity gives a refinement of Grothendieck-Riemann-Roch theorem on the level of differential forms. In \cite{Freix12} Freixas-Burgos-Liţcanu gave an axiomatic definition of those torsion forms and later used this result in \cite{Freix14} to generalize the arithmetic Grothendieck-Riemann-Roch theorem. See \cite{MaillovGreen} and \cite{GilSoul92}, \cite{GilSoRo08} for another interesting applications of torsion forms in Arakelov geometry.
	\par Puchol in \cite{Puchol2016} obtained the first term of the asymptotic expansion of $\mathcal{T}(\omega_{\text{fam}},F_p)$ when $F_p$ is the direct image of the sheaf associated to the increasing powers $p$ of a line bundle, which is positive along fibers. The main result of Bismut-Vasserot in \cite{BVasSym} follows from considering the direct image of the canonical line bundle on the projective fibration associated to the vector bundle $F_p$ on a ``family'' of manifolds over a point. In \cite[\S 3]{BismutBC13}, Bismut generalized the definition of torsion forms to the case of a holomorphic fibration (which is not necessarily Kähler). It is natural to expect that one can combine our result with \cite[\S 3]{BismutBC13} and \cite{Puchol2016} to get a general asymptotic expansion of the torsion forms for a holomorphic fibration. However in this paper we only work with the analytic torsion under the assumptions of Bismut-Vasserot in \cite{BVas}. We hope, in this way we can present clearly the ideas and avoid to introduce the sophisticated techniques as Toeplitz operators (cf. \cite[\S 7]{MaHol}), Bismut superconnection \cite{Bis86}, etc. We hope to come back to the general case very soon.
	\par A similar question in realms of the real analytic torsion was considered in \cite{MulAsympt}, \cite{BisMaZha17}. See also \cite{Bruning12}, \cite{Brav08} for related topics. For the analytic torsion on orbiolds, see \cite{MaOrbif2005}, \cite{Freixas16}. See \cite{Yoshi2004}, \cite{YoshiK3} for the application of the analytic torsion to the moduli space of K3 surfaces and \cite{MaillotBCOV} for the application in Calabi-Yau theefolds. There are many applications of the analytic torsion in Arakelov geometry, see \cite{KohRoss01} and later works of these authors, where they proved Lefschetz fixed point formula in Arakelov geometry. The results on the equivariant analytic torsion play an important role in their proof.
	\par This article is organized as follows. 
		In Section 2 we recall some properties of the Mellin transform and the definition of the holomorphic analytic torsion. We give a proof of Theorem \ref{thm-gen_asympt_zet}, relying on some technical tools, which we prove later in Section 3. In Section 3 we also explain some facts about diagonal and off-diagonal expansion of the heat kernel of the operator $\laplcomp^{L^p \otimes E} / p$.
		In Section 4 we prove Theorem \ref{dzeta_asympt}, we compare it with \cite{KlMa} and we give a relation to the arithmetic Riemann-Roch theorem.
		In Section 5 we recall the basics of the orbifolds, we prove Theorem \ref{thm-gen_asympt_orbi} and we describe a connection between Theorem \ref{thm-gen_asympt_orbi} and \cite{Hsiao16}. 
	\par {\bf{Notation.}} In this article denote by $\nat^*$ the set $\nat \setminus \{0\}$, by $T^{(1,0)}M$ the holomorphic tangent bundle of $M$ (see \S \ref{sect_2}) and by $T^{(0,1)}M := \overline{T^{(1,0)}M}$ the antiholomorphic tangent bundle,
	\begin{align*}
		&T^{*(0,1)}M = (T^{(0,1)}M)^*,  \qquad 
		&&\dfor[(0,j)]{M, E} = \mathscr{C}^{\infty} \big( M, \Lambda^{j} (T^{*(0,1)}M) \otimes E \big), \\
		&\dfor[(0,\bullet)]{M, E} = \oplus \dfor[(0,j)]{M, E},   \qquad 
		&&\dfor[(0,>0)]{M, E} = \oplus_{j > 0} \dfor[(0,j)]{M, E} . 
	\end{align*}
		Let $N$ be the number operator on the $\integ$-graded vector space $\dfor[(0, \bullet)]{M, E}$, i.e.
		\begin{equation}\label{defn_number_oper}
			N \cdot \alpha = j \alpha, \qquad \qquad \alpha \in \dfor[(0, j)]{M, E}.
		\end{equation}
		This induces a $\integ_2$-grading $\epsilon = (-1)^N$ on $\dfor[(0,\bullet)]{M, E}$. In general, let $A$ be an operator which acts on $\integ_2$-graded vector space $(V, \epsilon)$, its supertrace is defined as $\str{A} = \tr{\epsilon A}$. Sometimes, to make things more precise, we denote its trace/supertrace by ${\rm{Tr}}^{V} [A], {\rm{Tr}_s}^{V}[A]$.
	\par {\bf{Acknowledgements.}} This work is part of our PhD. thesis, which was done at Université Paris Diderot. We would like to express our deep gratitude to our PhD advisor Xiaonan Ma for his overall guidance, constant support and important remarks on the preliminary version of this article. 

\section{Asymptotics of heat kernels, Theorem \ref{thm-gen_asympt_zet}}\label{sect_2}
	This is an introductory section. In Section 2.1 we recall the definition of the holomorphic analytic torsion. In Section 2.2 we recall some machinery for studying it and we give a proof of Theorem \ref{thm-gen_asympt_zet}. Compared to \cite{BVas} and \cite[\S 5.4]{MaHol}, the major contribution of this section is Proposition \ref{prop-streng_infty}. 
	\subsection{Holomorphic analytic torsion}
		Before explaining our geometric situation, let's recall the Mellin transform:
		\begin{defn}[The Mellin transform]
		Let $f \in \cinf{]0, +\infty[}$ satisfies the following assumptions
		\par 1. There exists $ m \in \nat$ such that for any $ k \in \nat$, there is an asymptotic expansion as $t \to +0$
			\begin{equation}\label{eqn_mell_1}
				\textstyle f(t) = \sum_{i = -m}^{k} f_i t^i + o(t^k),
			\end{equation}
		\par 2. There are $\lambda, C > 0$ such that for $t \gg 1$
			\begin{equation}\label{eqn_mell_2}
				| f(t) | \leq C e^{-t \lambda}.
			\end{equation}		
		The Mellin transform of $f$ is the function $\mell{f}$, defined on the complex half-plane $\Re z > m$ by 
		\begin{equation}\label{defn_mell_form}
			\textstyle \mell{f}(z) :=  \frac{1}{\Gamma(z)} \int_0^{+ \infty} f(t) t^{z-1} \,dt.
		\end{equation}
		\end{defn}
		\noindent It is well-known that $\mell{f}$ extends holomorphically around $0$, and we have (cf. \cite[Lemma 9.35]{BGV})
			\begin{equation}\label{eq_mell_transform}
				\begin{aligned}
				& \textstyle \mell{f}(0) = f_0, \\
				& \textstyle \mell{f}'(0) = \int_0^{1} \big( f(t) -  \sum_{i = -m}^{0} f_i t^i \big) \frac{\,dt}{t} + 
				\int_1^{+ \infty} f(t) \frac{\,dt}{t} +  \sum_{i = -m}^{-1} \frac{1}{i}f_i - \Gamma'(1) f_0.
				\end{aligned}
			\end{equation}
		\begin{notat}\label{not_coefficients}
			Let's suppose that a function $f : ]0, +\infty[ \to \real$ satisfies (\ref{eqn_mell_1}). We denote $f_i$ by $f^{[i]}$.
		\end{notat}	
		\par 	Now let's recall the main object of this article: the analytic torsion.
		Let $(M, J)$ be a complex manifold with complex structure $J$. Let $g^{TM}$ be a Riemannian metric on $TM$ compatible with $J$, and let $\Theta = g^{TM}( J \cdot, \cdot) $ be the associated $(1,1)$-form. We call $(M, g^{TM}, \Theta)$ a Hermitian manifold. 
		\par Let $(M, g^{TM}, \Theta)$ be a compact Hermitian manifold of complex dimension $n$. The Riemann volume form $d v_M$ is given by
		\begin{equation}\label{defn_vol_form}
			d v_{M} := \tfrac{1}{n!}\Theta^n.
		\end{equation}
		Let's denote by $r^M$ the scalar curvature of $g^{TM}$ and by $\scal{\cdot}{\cdot}$ the $\comp$-linear extension of $g^{TM}$ to $TM \otimes_{\real} \comp$. We denote by $T^{(1,0)}M$ the $i$-eigenspace of $J \in \enmr{TM \otimes_{\real} \comp}$ and by $T^{(0,1)}M$ the $-i$-eigenspace. Then $g^{TM}$ induces a Hermitian metric $h^{T^{(1,0)}M}$ on $T^{(1,0)}M$ by the isomorphism $X \mapsto (X - i J X)/\sqrt{2}, X \in TM$. Let's denote by $R^{\det}$ the curvature of the Chern (Hermitian holomorphic) connection over $(\det T^{(1,0)}M, h^{\det})$, where $h^{\det}$ is the Hermitian metric on $\det T^{(1,0)}M$ induced by $h^{TM}$. In other words,
		\begin{equation}\label{defn_R_det}
			R^{T^{(1,0)}M} = (\nabla^{T^{(1,0)}M})^2, \quad	R^{\det} = \tr{R^{T^{(1,0)}M}},
		\end{equation}
		where $\nabla^{T^{(1,0)}M}$ is the Chern connection on $(T^{(1,0)}M, h^{T^{(1,0)}M})$.
		\par Now, let $E$ be a holomorphic vector bundle on $M$ with a Hermitian metric $h^E$. We call $(E, h^E)$ a holomorphic Hermitian vector bundle. We denote by $\nabla^E$ its Chern connection and by $R^E = (\nabla^E)^2$ its curvature.
	\par Let's denote by $\scal{\cdot}{\cdot}_{L^2}$ the $L^2$-scalar product on $\dfor[(0,\bullet)]{M,E}$, defined by
	\begin{equation}
		\scal{\alpha}{\alpha'}_{L^2} := \tinyint_M \scal{\alpha}{\alpha'}_h(x) \, d v_M(x), \quad \text{for any} \quad \alpha, \alpha' \in \dfor[(0,\bullet)]{M,E}, 
	\end{equation}
	where $\scal{\cdot}{\cdot}_h$ is the pointwise Hermitian product on $\Lambda (T^{*(0,1)}M) \otimes E$, induced by $h^{T^{(1,0)}M}$ and $h^E$.
	\par Let $\dbar{}^{E}$ be the Dolbeaut operator acting on the Dolbeaut complex $\dfor[(0,\bullet)]{M,E}$. We denote by $\dbar{}^{E*}$ the formal adjoint of $\dbar{}^{E}$ with respect to $\scal{\cdot}{\cdot}_{L^2}$. The Kodaira Laplacian is given by
	\begin{equation}
		\laplcomp^E := \dbar{}^{E} \, \dbar{}^{E*} + \dbar{}^{E*} \dbar{}^{E}.
	\end{equation}
	The operator $\laplcomp^E$ preserves the $\integ$-grading on $\dfor[(0,\bullet)]{M,E}$. 
	We also define
	\begin{equation}\label{defn_Dir_op}
		D^E := \sqrt{2}(  \dbar{}^{E} + \dbar{}^{E*} ), \quad \text{ then } \quad ( D^{E} )^2 = 2 \laplcomp^E.
	\end{equation}
		By Hodge theory, the operator $\laplcomp^E$ has finite dimensional kernel. We denote by $P$ the orthogonal projection onto this kernel and by $P^{\perp}= {\rm Id} - P$ the orthogonal projection onto its orthogonal complement.
		By the standard facts on heat kernels (see \cite[Theorem 2.30, Proposition 2.37]{BGV}), we can define the zeta-function:
		for $z \in \comp, \Re z > n$ we set
		\begin{equation}\label{defn_zeta_fun}
			\textstyle \zeta_E(z) :=  - \mell{\str{N \exp(-u \laplcomp^E)P^{\perp}}}.			
		\end{equation}
		\begin{defn}\label{defn-RS_anal_torsion}
			The analytic torsion of Ray-Singer of $(E, h^E)$ is defined as
			\begin{equation}
			\textstyle T(g^{TM}, h^{E}) := \exp  \left( - \tfrac{1}{2} \zeta_E'(0) \right).
			\end{equation}
		\end{defn}
		\begin{rem}
			Let $\det (\laplcomp^E|_{\Omega^i})$ be the regularized determinant of $\, \laplcomp^E|_{\dfor[(0,i)]{M}}$, then 
			\begin{equation}
				\textstyle T(g^{TM}, h^{E}) = \prod_i \det \big(\laplcomp^E|_{\Omega^i}\big)^{-(-1)^i i/2}.
			\end{equation}
		\end{rem}
	\subsection{Asymptotics of the analytic torsion on manifolds}\label{subsec_idea}
		In this section we present a proof of Theorem \ref{thm-gen_asympt_zet}. We follow closely the strategy of the proof of the main theorem in \cite{BVas} and we defer the proof of some technical details to Section \ref{subsec_proof_aux}.
		\par Let $(M, g^{TM}, \Theta)$ be a compact Hermitian manifold and let $(E, h^E)$, $(L, h^L)$ be holomorphic Hermitian vector bundles over $M$. We suppose that $(L, h^L)$ is a positive line bundle, i.e. 
		\begin{equation}\label{eqn_posit}
			R^L(U, \overline{U}) > 0, \quad \text{for any} \quad U \in T^{(1,0)}M. 
		\end{equation}
	We denote by $\laplcomp_p$ the Laplacian associated to $L^p \otimes E$ and by $\zeta_p, p \in \nat$ the zeta-function $\zeta_{L^p \otimes E}$.	 For $x, y \in M $, we denote by $\exp ( - u \laplcomp_p / p )(x,y)$ the smooth kernel with respect to the volume form $dv_M$ of the heat operator $\exp ( - u \laplcomp_p / p )$.
	\begin{sloppypar}
	\begin{thm}[{\cite[Theorem 4]{BVas}, \cite[Theorem 1.2]{MaBerg06}}]\label{thm-gen_asympt_exp}
		There are smooth sections $a_{i,u}(x)$, $i \in \nat$ of $\oplus_{l \geq 0}\enmr{\Lambda^{l}(T^{*(0, 1)}M)  \otimes E}$ over $M$ such that for every $u > 0$, we have
		\begin{equation}
			\textstyle \exp ( - u \laplcomp_p/p )(x,x) = \sum_{i = 0}^{k} a_{i,u}(x) p^{n-i} + O(p^{n-k-1}), \quad \text{as} \quad p \to +\infty,		
		\end{equation}
		and the estimate is uniform in $x \in M$ and $u$, as $u$ varies in a compact subspace of $]0, + \infty[$.
	\end{thm}
	\end{sloppypar}
	For the proof of the following proposition see Section \ref{subsec_proof_aux}.
	\begin{prop}\label{exp_a_i_u}
		There are smooth sections $a_i^{[j]}(x)$ of  $\oplus_{l \geq 0} \enmr{\Lambda^{l} (T^{*(0,1)}M) \otimes E}$ such that 
		\begin{equation}\label{exp_a_i_u_zero}
			\textstyle a_{i,u}(x) = \sum_{j = -n}^{k} a_i^{[j]}(x) u^j + o(u^k),
		\end{equation}
		as $u \to 0$, for any $k \in \nat$. Moreover, there are $c_i, d_i > 0$ such that for any $u \gg 1, x \in M$
		\begin{equation}\label{exp_a_i_u_infty}
			\big| a_{i,u}^{[>0]}(x) \big| \leq c_i \exp(- d_i u ),
		\end{equation}
		where $[>0]$ means the projection onto positive degree terms.
	\end{prop}
	The estimation (\ref{exp_a_i_u_infty}) was proved in \cite[Theorem 1.2]{MaBerg06}.
	Now we can restate Theorem \ref{thm-gen_asympt_zet} in a precise way
	\begin{thm}\label{thm-gen_asympt_zet_gen}
		There are local coefficients $\alpha_i, \beta_i \in \real, i \in \nat$ such that for any $k \in \nat$, as $p \to + \infty$
		\begin{equation}
			\textstyle \zeta_{p}'(0) = \sum_{i = 0}^{k} p^{n-i} \left( \alpha_i \log p + \beta_i \right) + o(p^{n-k}),
		\end{equation}
		as $p \to \infty$, where
		\begin{equation}
		\begin{aligned}\label{thm-a_i-b_i_formula}
			\alpha_i &= \textstyle \int_M \str{N a_i^{[0]}(x)} \, dv_M(x),  \quad
			\beta_i &=  \textstyle -{\rm{M}}_u \big[ \int_M \str{N a_{i,u}(x)} \, dv_M(x) \big]'(0).
		\end{aligned}
		\end{equation}
	\end{thm}
	To prove Theorem \ref{thm-gen_asympt_zet_gen}, we need to introduce the constants $b_{p,i} \in \real$ for $i \geq -n, p \in \nat^*$, which satisfy the following asymptotic expansion for any $k \in \nat$ (cf. \cite[Theorem 2.30]{BGV})
	\begin{equation}\label{defn-b_pj}
		\textstyle p^{-n} \str{N \exp ( - u \laplcomp_p / p) } = \sum_{i=-n}^{k} b_{p,i} u^i + o(u^{k+1}), \quad \text{as} \quad u \to +0 .
	\end{equation}	
	We also need the next three propositions, for their proof see Section \ref{subsec_proof_aux}.
	\begin{prop}\label{prop_b_p_i_exp}
		As $p \to \infty$, the following expansion holds for any $k \in \nat$
		\begin{equation}\label{eqn_b_i^j_as_der}
			\textstyle b_{p,i} = \sum_{j=0}^{k} b_i^{[j]}p^{-j} + o(p^{-k}), \quad \text{with} \quad \textstyle b_i^{[j]} = \int_M \str{N a_{j}^{[i]}(x)} \, dv_M(x).
		\end{equation}
	\end{prop}
	\noindent The following propositions are essential extensions of \cite[Theorem 2]{BVas} (cf. \cite[\S 5.5]{MaHol}). They form the core of the proof.
	\begin{prop}\label{prop-streng_zero}
		For any $k\in \nat, u_0 > 0$ there exist $C > 0$ such that for any $u \in ]0, u_0[, p \in \nat^*$:
		\begin{multline}
		 p^k  \Big|
				\Big( 
					p^{-n} \str{N \exp ( - u \laplcomp_p / p )} 
					-\sum_{j=-n}^{0} u^j b_{p,j}
				\Big) \\
				-
				  \sum_{i=0}^{k-1} 
				 p^{-i}
				\Big( 
					 \int_M \str{N a_{i,u}(x)} \, dv_M(x)
					- \sum_{j=-n}^{0} u^j b_j^{[i]}
				\Big) \Big| \leq C u. 
		\end{multline}
	\end{prop}
	\begin{prop}\label{prop-streng_infty}
		For any $k \in \nat, u_0 > 0$ there are $c,C > 0$ such that for $u > u_0, p \in \nat^*$:
		\begin{equation}\label{prop-streng_infty_eqn}
			 p^k \Big| p^{-n} \str{N \exp ( - u \laplcomp_p / p )} - 
		\sum_{j=0}^{k-1} p^{-j} \int_M \str{N a_{j,u}(x)} \, dv_M(x) \Big| 
		\leq C \exp(-cu).
		\end{equation}
	\end{prop}
	We point out that both of those Propositions are obtained for $k=0$ in \cite[Theorem 2]{BVas}. The proof of Proposition \ref{prop-streng_zero} for any $k$ is more-or-less parallel to the case $k = 0$. However, in Proposition \ref{prop-streng_infty}, the original spectral gap approach works only for $k = 0$. 
	\begin{proof}[Proof of Theorem \ref{thm-gen_asympt_zet_gen}.]	
		We introduce the function 
		\begin{equation}
			\tilde{\zeta_p}(z) = p^{z-n} \zeta_p(z).
		\end{equation}
		It satisfies the following
		\begin{equation}\label{form-zeta'}
			p^{-n} \zeta'_p(0) = - \log (p)   \tilde{\zeta_p}(0) +  \tilde{\zeta_p}'(0), 
		\end{equation}
		\begin{equation}\label{zeta_tilde_mellin}
		\textstyle
			\tilde{\zeta_p}(z) 
			= 
			 - p^{-n} {\rm{M}}_u \big[\str{N \exp (- u \laplcomp_p / p ) } \big](z).
		\end{equation}
		
		We remark that Theorem \ref{thm-gen_asympt_zet_gen} “follows" formally from Theorem \ref{thm-gen_asympt_exp}, (\ref{eq_mell_transform}), (\ref{form-zeta'}) and (\ref{zeta_tilde_mellin}). Now we are going to make this reasoning precise.
		
		Using (\ref{eq_mell_transform}) and (\ref{zeta_tilde_mellin}), we obtain
		\begin{align}
			 \tilde{\zeta_p}'(0) = & - \int_0^{1} \Big( p^{-n} \str{N \exp (  - u\laplcomp_p /p )} - \sum_{j=-n}^{0} b_{p,j} u^j  \Big) \frac{\,du}{u} \nonumber \\
			& - \int_1^{+\infty} p^{-n} \str{N \exp ( - u \laplcomp_p / p )} \frac{\,du}{u} -
			 \sum_{j=-n}^{-1} \frac{b_{p,j}}{j} + \Gamma'(1)b_{p,0}, \label{form-tilde_zeta'_0} \\
			  \tilde{\zeta_p}(0) = &- b_{p,0}. \label{form-tilde_zeta_0}
		\end{align}								
		The following notation makes sense due to Proposition \ref{exp_a_i_u}:
		\begin{equation}\label{nu_i_formula}
			\textstyle \nu^{[i]} = - {\rm{M}}_u \big[ \int_M \str{N a_{i,u}(x)} \, dv_M(x) \big]'(0).  
		\end{equation}
		By (\ref{eq_mell_transform}) and (\ref{eqn_b_i^j_as_der}), we have
		\begin{multline}\label{nu_i_expanded}
			 \nu^{[i]} = - \int_0^{1} \Big(  \int_M \str{N a_{i,u}(x)} \, dv_M(x)
			- \sum_{j=-n}^{0} u^j b_j^{[i]}  \Big) \frac{\,du}{u} \\
			 - 
			\int_1^{+\infty} \int_M \str{N a_{i,u}(x)} \, dv_M(x) \frac{\,du}{u} -
			 \sum_{j=-n}^{-1} \frac{1}{j}b_j^{[i]} + \Gamma'(1)b_0^{[i]}.
		\end{multline}
		Suppose that the following limit holds for any $k \in \nat$
			\begin{equation}\label{nu_i_limit}
				\textstyle \lim_{p \to +\infty} p^{k} \big( \tilde{\zeta_p}'(0) - \sum_{i = 0}^{k-1} \nu^{[i]}p^{-i} \big) = \nu^{[k]}.
			\end{equation}
		Then from  (\ref{form-zeta'}), (\ref{form-tilde_zeta'_0}), (\ref{form-tilde_zeta_0}), (\ref{nu_i_formula}), (\ref{nu_i_limit}) and Proposition \ref{prop_b_p_i_exp}, we obtain Theorem \ref{thm-gen_asympt_zet_gen}.
		\par Now let's prove (\ref{nu_i_limit}). By (\ref{form-tilde_zeta'_0}) and (\ref{nu_i_expanded}) it suffices to prove that for $k \in \nat$, as $p \to \infty$,
			\begin{align}
				&  1) \int_0^{1} p^k \Big(  
					\Big( 
						p^{-n} \str{N \exp ( - u \laplcomp_p / p )} 
						-\sum_{j=-n}^{0} u^j b_{p,j} 
					\Big) \nonumber \\ \nonumber
				&  \qquad \qquad -
					 \sum_{i=0}^{k-1} 
						 p^{-i}
						\Big( 
				 			\int_M \str{N a_{i,u}(x)} \, dv_M(x)
							- \sum_{j=-n}^{0} u^j b_{j}^{[i]}
						\Big) \Big) \frac{du}{u} \\ 
				&   \qquad \qquad  \qquad \qquad \to \int_0^{1}  \Big(  \int_M \str{N a_{k,u}(x)} \, dv_M(x)
					- \sum_{j=-n}^{0} u^j b_j^{[k]}  \Big) \frac{\,du}{u}, \\ 
				&  2) \int_1^{+\infty} p^k \Big( p^{-n} \str{N \exp ( - u \laplcomp_p / p )} - 
					\sum_{j=0}^{k-1} p^{-j} \int_M \str{N a_{j,u}(x)} \, dv_M(x) \Big) \nonumber  \\ 
		 		&   \qquad \qquad \to
					\int_1^{+\infty} \int_M \str{N a_{k,u}(x)} \, dv_M(x) \frac{\,du}{u}, \\ 
				&  3) p^k \Big( b_{p,j} - \sum_{i=0}^{k-1}  b_j^{[i]}p^{-i}  \Big) \to b_j^{[k]}.
			\end{align}
		The first and second limits are consequences of Lebesgue dominated convergence theorem and Propositions \ref{prop-streng_zero}, \ref{prop-streng_infty} correspondingly. The third one is a consequence of Proposition \ref{prop_b_p_i_exp}.
		\par \textbf{Now, we will prove that $\alpha_i, i \in \nat$ do not depend on $g^{TM}, h^{L}, h^E$.} Let $c \in \real \to g_c^{TM}, h^{L}_c, h^{E}_c$ be some variations of the metrics on $TM, L, E$. We suppose that $g_c^{TM}$ is compatible with the complex structure $J$ of $M$. We denote by $*_c$ the Hodge-star operator associated to $g^{TM}_{c}$ and by $\laplcomp_{p, c}$ the Kodaira Laplacian, associated to $g_c^{TM}, h^{L}_c, h^{\xi}_c$. From \cite[Theorems 1.18]{BGS3}, there are constants $M_{j,c}^{p}$, $j \geq -1$, $p \in \nat^*$ such that for any $k \in \nat$, we have
		\begin{equation}\label{eqn_top_const}
			- {\rm{Tr}_s} \Big[\Big( (*_c)^{-1} \frac{\partial *_c}{\partial c} + p (h^L_{c})^{-1} \frac{\partial h^{L}_{c}}{\partial c} +  (h^E_{c})^{-1} \frac{\partial h^{E}_{c}}{\partial c} \Big) \exp(- u \laplcomp_{p, c} / 2)\Big] = \sum_{j=-1}^{k} M_{j,c}^{p} u^j + o(u^k).
		\end{equation}
		Now, from \cite[(1.117)]{BGS3}, we have
		\begin{equation}\label{tors_anom}
			 - 2 \frac{\partial}{\partial c} \log T(g^{TM}_{c}, h^{L^p \otimes E}_{c}) = - M_{0,c}^{p} + {\rm{Tr}_s}  \Big[ (*_c)^{-1} \frac{\partial *_c}{\partial c} P_{c} \Big],
		\end{equation}
		where $P_c$ is the orthogonal projection onto $\ker ( \laplcomp_{p, c})$ with respect to $g_c^{TM}, h^{L}_c, h^{E}_c$.
		We remark that
		\begin{equation}\label{eqn_top_const}
			- {\rm{Tr}_s} \Big[\Big( (*_c)^{-1} \frac{\partial *_c}{\partial c} + p (h^L_{c})^{-1} \frac{\partial h^{L}_{c}}{\partial c} +  (h^E_{c})^{-1} \frac{\partial h^{E}_{c}}{\partial c} \Big) \exp(- u \laplcomp_{p, c} / 2p)\Big] = \sum_{j=-1}^{k} M_{j,c}^{p}  p^{-j} u^j + o(u^k).
		\end{equation}
		Now, from (\ref{eqn_top_const}) we see that, similarly to Proposition \ref{prop_b_p_i_exp}, $M_{0,c}^{p}$ has an asymptotic expansion of the form (\ref{eqn_b_i^j_as_der}), as $p \to \infty$. From \cite[Theorem 4.1.1]{MaHol} we see that the asymptotics of $\str{(*_c)^{-1} \tfrac{\partial *_c}{\partial c} P_{c}}$ contains only powers of $p$. Thus, the change of the metric doesn't affect $\alpha_i$, since only the powers of $p$ appear in the asymptotics of (\ref{tors_anom}).
	\end{proof}
	
	\section{Heat kernel of the high power of positive line bundle}\label{sect_prelim}
		Here we recall some fundamental results about the asymptotic expansion of the heat kernel of $\laplcomp_p/p$. 
		For this we use the localization procedure of \cite[\S 2]{MaDinhPac}, \cite[\S 3 .4]{MaMar08a}. In our context this procedure is more natural than the one from \cite{MaHol}, \cite{MaBerg06} since it respects the degrees of differential forms. This property permits us to give simple proofs of long-time estimates on the heat kernel (see Theorems \ref{thm-exp_proj_expansion}, \ref{thm_L_2_t_off_diag_u_infty}). Certainly,  the original localization procedure from \cite{MaHol}, \cite{MaBerg06} also gives the final result, but then one has to inevitably use some results on the Bergman kernel.
		\par This section is organized as follows. In Section 3.1 we recall how to localize the calculation of the asymptotic expansion and how to tackle this localisation. Almost all the results of Section 3.1 appeared in \cite{MaHol} and were inspired by \cite{BisLeb91}. In Section 3.2 we recall the off-diagonal expansion of the heat kernel of the local version of the operator $\laplcomp_p/p$. Finally, in Section 3.3 we prove Propositions \ref{exp_a_i_u}, \ref{prop_b_p_i_exp}, \ref{prop-streng_zero}, \ref{prop-streng_infty}; thus, completing the proof of Theorem \ref{thm-gen_asympt_zet_gen}.

\subsection{Localization of the asymptotic expansion of the heat kernel}
	\begin{sloppypar}
		In this section we recall a localization procedure from \cite{MaDinhPac} of the asymptotic expansion of $\exp ( - u \laplcomp_p / p)(x,x), x \in M$ as $ p \to +\infty$. We conserve the notation from Section \ref{sect_2}.
	\end{sloppypar}
	\begin{sloppypar}
	To work with non Kähler metrics we recall the definition of \textit{Bismut connection}. Let $(X, g^{TX}, \Theta_X)$ be a Hermitian manifold. 
	Let $S^B$ be a $1$-form with values in the antisymmetric elements of $\enmr{T^{(1,0)}X}$, which satisfies (see \cite[Definition 1.4]{BisNonKah})
		\begin{equation} 
			\scal{S^B(U)V}{W} = \tfrac{1}{2} \imun \big( (\partial - \dbar) \Theta_X \big)(U,V,W).
		\end{equation}
		\begin{defn}[{\cite[(1.15)]{BisNonKah}, cf. also \cite[Definition 1.2.9]{MaHol}}]
			The Bismut connection $\nabla^B$ on $TX$ is defined by $\nabla^B = \nabla^{TX} + S^B$, where $\nabla^{TX}$ is the Levi-Civita connection on $(TX, g^{TX})$.
		\end{defn}
		The connection $\nabla^B$ preserves the complex structure of $TX$. Its family version was also defined by Bismut in \cite[\S 3.6 and Theorem 3.8.1]{BismutBC13}.
	\end{sloppypar}
	\begin{thm}[Bismut-Vasserot {\cite[Theorem 1]{BVas}}]\label{thm-lapl_spectral_gap}
		There exists $c > 0$ such that
		$$\spec ( \laplcomp_{p}) \subset \{ 0 \} \cup [cp, + \infty[, \qquad  \ker (\laplcomp_{p}) \subset \dfor[(0,0)]{M, L^p \otimes E}, \quad \text{for $p \gg 1$.}$$
	\end{thm}
	\par For $e = v^{(1,0)} + v^{(0,1)} \in T^{(1,0)}M \oplus T^{(0,1)}M  = TM \otimes_{\real} \comp$ we denote by $c(e)$ the operator on $\dfor[(0, \bullet)]{M}$, defined by
		\begin{equation}
			c(e) = \sqrt{2} ( \overline{v}^{(1,0), \ast} \wedge - i_{v^{(0,1)}} ),
		\end{equation}
	where $\wedge$ and $i$ are the exterior and interior product respectively. Let $e_1, \ldots, e_{2n}$ be an orthonormal frame of $(TM, g^{TM})$ and $e^1, \ldots, e^{2n}$ its dual frame.
	We define 
	\begin{equation}\label{defn_c_op}
		{}^{c}(e^{i_1} \wedge e^{i_2} \wedge  \cdots \wedge e^{i_j}) = c(e_{i_1}) c(e_{i_2}) \cdots c(e_{i_j}),
	\end{equation}
	for $0 < i_1 < \ldots < i_j \leq n$. We extend this operation $\comp$-linearly for any $B \in \Lambda^{\bullet}(T^*M \otimes_{\real} \comp) $.
	\par We take $x \in M$. 
	Let $\psi : M \supset U \to V \subset \comp^{n}$ be a holomorphic local chart such that $B(0, 4 \epsilon) \subset V$ and the restriction of $E$ over $U$ is trivial, where
		\begin{equation}\label{eqn_inj_rad}
			0 < \epsilon < r_{\rm{inj}} / 4, \text{ where $r_{\rm{inj}}$ is the injectivity radii of M }. 
		\end{equation}
		We denote in the sequel $X_0 := T_{x}M \simeq U$. We denote by $\rho : \real \to [0,1]$ a smooth positive function such that
		\begin{equation}\label{defn_rho_fun}
			\rho(u) =
  			\begin{cases} 
      			\hfill 0,  & \text{ for  $|u| > 4$,} \\
      			\hfill 1,  & \text{ for  $|u| < 2$.} \\
 			\end{cases}
		\end{equation}
		We define a Riemannian metric $g^{TX_0}(Z) = g^{TM}(\rho(|Z|/\epsilon)Z)$ over $X_0$. We choose a holomorphic frame of $E$ over $U$ and we introduce the Hermitian product $h^{E_0}$ on $E_0 = X_0 \times E_x$ over $X_0$ by $h^{E_0}(Z) = h^{E}(\rho(|Z|/\epsilon) Z)$, where $h^{E}$ is the matrix of the Hermitian product in the chosen holomorphic frame. Then $g^{TX_0}, h^{E_0}$ coincide with $g^{TM}$ and $h^{E}$ over $B(0,2\epsilon)$ and with trivial structures $g^{TM}_{x}, h^{E}_{x}$ away from $B(0, 4 \epsilon)$. We denote by $R^{E_0}$ the Chern curvature of $(E_0, h^{E_0})$ and by $\Theta_0$ the Hermitian form associated to $g^{TX_0}$. 
		
		\par Let $\sigma$ be a holomorphic frame of $L$ over $U$. It defines a trivialisation $\psi: L|_U \to U \times \comp$. We define a function $\phi(Z), Z \in X_0$ by $e^{-2 \phi(Z)} = |\sigma|_{h^{L}}^{2}(Z)$. Let's denote by $\phi^{[1]}$ and $\phi^{[2]}$ the first and second order Taylor expansions of $\phi$ at $x$, i.e.
		\begin{align}
			& \phi^{[1]}(Z) = \sum_{j=1}^{n} \Big( \frac{\partial \phi}{\partial z_j}(x) z_j +  \frac{\partial \phi}{\partial \overline{z}_j}(x)  \overline{z}_j  \Big), \\
			&  \phi^{[2]}(Z) = \Re \Big( \sum_{j, k=1}^{n} \Big( \frac{\partial^2 \phi}{\partial z_j \partial z_k}(x) z_j z_k + \frac{\partial^2 \phi}{\partial z_j \partial \overline{z}_j}(x) z_j \overline{z}_k  \Big) \Big),
		\end{align}
		where $(z_1, \ldots, z_n)$ are the complex coordinates of $Z$.
		We define a function $\phi_{\epsilon}(Z)$ over $X_0$ by
		\begin{equation}
			\phi_{\epsilon}(Z) = \rho ( |Z| / \epsilon ) \phi(Z) + (1 -  \rho ( |Z|/\epsilon ) ) \big( \phi(x) + \phi^{[1]}(Z) + \phi^{[2]}(Z) \big).
		\end{equation}
		Let $h_{\epsilon}^{L_0}$ be the metric on $L_0 := X_0 \times \comp$ defined by 
		$ |1|^2_{h_{\epsilon}^{L_0}} = e^{- 2 \phi_{\epsilon}(Z)} $.
		Let $\nabla^{L_0}$ be the Chern connection on $(L_0, h^{L_0}_{\epsilon})$ and let $R^{L_0}_{\epsilon}$ be the curvature of it. Then by \cite[(2.28)]{MaDinhPac} 
		\begin{equation}\label{curv_positive}
			\text{$R^{L_0}_{\epsilon}$ is positive for $\epsilon > 0$ small enough.}
		\end{equation}				
		From now on, we fix $\epsilon > 0$ which satisfies (\ref{eqn_inj_rad}) and (\ref{curv_positive}). 
		We trivialize $L_0$ by a unitary section $S_{x}$ of $(L_0, h_{\epsilon}^{L_0})$, which we write as
		\begin{equation}
			S_{x} = e^{\tau} 1 \quad \text{with} \quad  \tau(0) = \phi(x),
		\end{equation}
		where the function $\tau$ is given by 
		$ \tau(Z) = \phi(x) - 2 \int_0^{1} (i_Z \partial \phi_{\epsilon})_{tZ} dt $
		, so that $S_x$ satisfies $\nabla^{L_0}_{Z}S_x = 0$. By abuse of notation, we drop $\epsilon$ from the notation introduced before.		
		\par We define a positive function $k : X_0 \to \real$ by the identity
		\begin{equation}\label{defn_k_fun}
			\,dv_{X_0}(Z) = k(Z) \,dv_{T_{x}M}(Z),
		\end{equation}
		where $d v_{T_{x}M}, \,dv_{X_0}(Z)$ are the Riemann volume forms on $X_0$, induced by $g^{TM}_{x}$ and $g^{TX_0}$ respectively. 
		We see, in particular, that $k(0) = 1$.
		We denote by $h^{\det_0}$ the Hermitian metric, induced on $\det T^{(1,0)} X_0$ by $g^{TX_0}$.
		\begin{sloppypar}	
		 We define a smooth self-adjoint section $\Phi_{E_0}$ of $\oplus_{i \geq 0} \enmr{\Lambda^{i}(T^{*(0,1)}X_0) \otimes E_0}$ over $X_0$ by
		\begin{equation}\label{defn_phi_e}
			\Phi_{E_0} := \tfrac{1}{4}r^{X_{0}} + \tfrac{1}{2}{}^{c}(R^{E_0} + R^{\det_0}) + \tfrac{1}{2}\imun {}^{c}(\dbar \partial \Theta_0) - \tfrac{1}{8} | (\dbar - \partial) \Theta_0 |^2,
		\end{equation}	
		where $r^{X_{0}}$ is the scalar curvature of $g^{TX_0}$ and $R^{\det_0}$ is the Chern curvature of $(\det T^{(1,0)} X_0,  h^{\det_0})$.
		\end{sloppypar}
		\begin{sloppypar}
			We denote by $\nabla^{B, \Lambda^{0, \bullet}_{0}}$ the natural extension of the Bismut connection $\nabla^B$ of $(X_0, g^{TX_0})$ on $\Lambda^{\bullet}(T^{*(0,1)}X_0)$ (see \cite[(1.4.27)]{MaHol}). We set
		\begin{equation}\label{defn-L_p}
			L_{p,{x}} := \Delta^{B, \Lambda^{0, \bullet}_{0} \otimes L_{0}^{p} \otimes E_{0} } +
		\tfrac{1}{2}p{}^{c}(R^{L_0})
		+ \Phi_{E_0},
		\end{equation}
		where $\Delta^{B, \Lambda^{0,\bullet}_{0} \otimes L_{0}^{p} \otimes E_{0}}$ is the Bochner Laplacian on $\Lambda^{\bullet}(T^{*(0,1)}X_0) \otimes L_{0}^{p} \otimes E_{0}$ associated with
		\begin{equation}\label{eqn_conn_product}
			\nabla^{B, \Lambda^{0, \bullet}_{0} \otimes L_{0}^{p} \otimes E_0} := \nabla^{B, \Lambda^{0,\bullet}_{0}} \otimes 1 \otimes 1
		+ 1 \otimes \nabla^{L_{0}^{p}} \otimes 1 + 1 \otimes 1 \otimes \nabla^{E_0},
		\end{equation}
		and $g^{TX_0}$, $h^{L_0}$, $h^{E_0}$.
		By the trivialization as above, we have $\Lambda^{\bullet}(T^{*(0,1)}X_0) \otimes L_{0}^{p} \otimes E_{0} \simeq \Lambda^{\bullet}(T^{*(0,1)}_{x}M) \otimes L_{x}^{p} \otimes E_{x}$.
		The operator $L_{p,{x}}$ preserves $\integ$-grading on $\dfor[(0, \bullet)]{X_0, L_{x}^{p} \otimes E_{x}}$ and the following formula holds  (see \cite[Theorem 1.3]{BisNonKah})
		\begin{equation}\label{eqn_L_p_full_sqr}
			L_{p, x} = 2 \big(\dbar^{X_0}_{p} + \dbar^{X_0*}_{p}\big)^2,
		\end{equation}
		where $\dbar^{X_0}_p$ is the Dolbeaut operator acting on $\dfor[(0, \bullet)]{X_0, L_0^{p} \otimes E_0}$, and $\dbar^{X_0*}_{p}$ is its adjoint with respect to the $L^2$-norm induced by $g^{TX_0}, h^{L_0}$ and $h^{E_0}$. Then $L_{p, x}$ is self-adjoint with respect to this norm.
		\end{sloppypar}
		\par All the constructions made here could be performed uniformly in a neighbourhood of $x \in M$. For the rest of this article we denote by $\mathscr{C}^{m}(M)$ the $\mathscr{C}^{m}$-norm with respect to the parameter $x$.
		\par By (\ref{eqn_L_p_full_sqr}) and the positivity of $R^{L_0}$, we have
		 \begin{thm}[{ \cite[Theorem 1]{BVas} cf. also the proof of \cite[Theorem 1.5.7, 1.5.8]{MaHol}}]\label{thm_spec_gap_L_p}
		 There is $\mu > 0$ such that
		 	$$\spec (L_{p,{x}}) \subset \{ 0 \} \cup [\mu p, + \infty[, \qquad \ker (L_{p,{x_0}}) \subset \dfor[(0, 0)]{X_0, L_{0}^{p} \otimes E_{0}}, \quad \text{for} \quad p \gg 1.$$
		 \end{thm}
		 \begin{sloppypar}
		 	For $(Z,Z') \in X_0 \times X_0$ we denote by $\exp ( -u L_{p,{x}})(Z,Z')$ the smooth kernel of the heat operator $\exp ( -u L_{p,{x}} )$ with respect to the volume form $\,dv_{X_0}$. We have
		 \end{sloppypar}
		 \begin{lem}[{{\cite[Lemma 1.6.5, (5.5.73)]{MaHol}}}]\label{lem-lapl_L_p-kernels}
			There are constants $C,c > 0, l \in \nat$ such that uniformly on $p \in \nat^*, u \in ]0,+\infty[, x \in M$ and $Z, Z' \in T_xM; |Z|, |Z'| < \epsilon$, we have
			\begin{equation}
				\Big|  \exp ( - u \laplcomp_p / p )(\exp_x(Z),\exp_x(Z')) - \exp ( - u L_{p,x} / (2p) )(Z, Z') \Big| 
		 	\leq C p^l \exp ( - c p / u ).
			\end{equation}
		\end{lem}
		\begin{rem}\label{lem_lapl_l_p_off_diag}
			Since the proof of Lemma \ref{lem-lapl_L_p-kernels} relies on finite propagation speed of solutions of the hyperbolic equations \cite[Theorem D.2.1]{MaHol}  and on the fact that $\laplcomp_p$ is essentially self-adjoint operator, and those facts hold for orbifolds (see \cite[p.230]{MaOrbif2005}),  Lemma \ref{lem-lapl_L_p-kernels} itself holds for orbifolds.
		\end{rem}

\paragraph{More properties of the asymptotics of heat kernel.}
	Now we recall a procedure of replacing a discrete parameter $p \in \nat$ in the construction of $L_{p,{x}}$ to a continuous $t \in [0,1]$. This permits us to interpret the asymptotic expansion in $p$ as an instance of a Taylor expansion in $t$.
	\begin{sloppypar}
	We recall that we fixed a unitary section $S_x$ of $L_{x}$, so that we can say that $L_{p,{x}}, p \in \nat^*$ act on $\dfor[(0,\bullet)]{X_0, E_{0}}$, which is independent of $p$. For $s \in \dfor[(0,\bullet)]{X_0, E_{0}}, Z \in X_0, t = 1/\sqrt{p}$, we set
		 \begin{equation}\label{defn-L_2^t}
		 	\begin{aligned}
		 		& (S_t s)(Z) := s(Z/t), \\
		 		& \nabla_t := S_t^{-1} t k^{1/2} \nabla^{B, \Lambda^{0,\bullet}_{0} \otimes L_0^{p} \otimes E_0} k^{-1/2} S_t, \\
		 		& L_{2, x}^{t} := S_t^{-1} t^2 k^{1/2} L_{p,{x}} k^{-1/2} S_t.
		 	\end{aligned}
		 \end{equation}
		 By \cite[(1.6.31)]{MaHol} the definition of $\nabla_t, L_{2, x}^{t}$ extends for $t \in ]0,1]$. Moreover, the operator $L_{2, x}^{t} $ is self-adjoint with respect to (\ref{sob_norm_defn}).
	\end{sloppypar}
	\begin{sloppypar}
		\par In what follows, we will repeatedly use the results from \cite{MaHol}. Their localization procedure is different from ours, but their arguments apply directly when one chooses the Sobolev norms as
		\begin{align}\label{sob_norm_defn}
			&  \norm{s}_{t, 0}^{2} := \int_{X_0} \norm{s(Z)}_{t, h}^{2} dv_{T_xM}(Z),\\
			&  \norm{s}_{t, m}^{2} :=  \sum_{k = 0}^{m} \sum_{i_1, \ldots, l_k = 1}^{2n} \norm{\nabla_{t, e_{i_1}} \ldots \nabla_{t, e_{i_k}}s }_{t,0}^{2},
		\end{align}
		where $s \in \dfor[(0, \bullet)]{X_0, E_0}$, $\norm{\cdot}_{t, h}(Z)$ is the pointwise norm induced by $g^{TX_0}(tZ), h^{E_0}(tZ)$,
		and $e_1, \ldots, e_{2n}$ are as in (\ref{defn_c_op}).
	\end{sloppypar}
		 \par Let  $w_1, \ldots, w_n$ be an orthonormal frame of $(T^{(1,0)}_{x}M, h^{T^{(1,0)}M}_{x})$ and let $w^1, \ldots, w^n$ be its dual frame.
		 We denote by $\nabla_{0, \cdot}$ the connection on the vector bundle $\Lambda^{\bullet} (T^{*(0,1)} X_0) \otimes E_{0}$ and by $L_{2, x}^{0}$ the operator on $\dfor[(0, \bullet)]{X_0, E_{0}}$, defined by the formulas
		 \begin{equation}\label{defn-L_2^0}
		 	\begin{aligned}
		 		& \nabla_{0, \cdot} &&:= \nabla_{\cdot} + \tfrac{1}{2} R^L_{{x}}(Z,\cdot), \\
		 		& L_{2,{x}}^{0} &&:= \textstyle - \sum_i \nabla_{0,e_i}^{2} + 2 \sum_{i,j}  R^L_{{x}}(w_i, \overline{w}_j) \overline{w}^j \wedge i_{\overline{w}_i} -  \sum_i R^L_{{x}}(w_i, \overline{w}_i).
		 	\end{aligned}
		 \end{equation}
		 \begin{lem}[{\cite[Lemma 1.6.6]{MaHol}}]
		 	The family of operators $\nabla_t,  L_{2,x}^{t}$ is smooth in $t$ and 
		 	\begin{equation}
		 		\nabla_{t} \to \nabla_0,
		 		\qquad
		 		L_{2, x}^{t} \to L_{2, x}^{0}, \quad \text{as} \quad t \to 0.
		 	\end{equation}
		 \end{lem}
		 We define the operators $\mathcal{O}_1, \mathcal{O}_2, \ldots$ by the following expansion, as $t \to 0$,
		 \begin{equation}\label{defn_O_i}
		 L_{2, x}^{t} = L_{2, x}^{0} + t \mathcal{O}_1 + t^2 \mathcal{O}_2  + \cdots + t^k \mathcal{O}_k  + o(t^k), \quad k \in \nat.
		 \end{equation}
		 We also denote by $\exp(u L_{2, x}^{t})(Z,Z')$ the smooth kernel of the heat operator $\exp(- u L_{2, x}^{t})$ with respect to $\, d v_{T_{x}M}$. Then for $t = 1 / \sqrt{p}$, we have (cf. \cite[(1.6.66)]{MaHol})
		 \begin{equation}\label{eqn-L_p_L^t_idty}
		 	\exp ( - u L_{p,{x}} / p )(Z, Z') = p^n \exp(-u L_{2,{x}}^{t}) (Z/t, Z'/t)k^{-1/2}(Z)k^{-1/2}(Z').
		 \end{equation}
		 Now we recall some properties of the operator $L_{2, x}^{t}$. The reason why we are interested in it is Lemma \ref{lem-lapl_L_p-kernels} and (\ref{eqn-L_p_L^t_idty}).  By \cite[(4.2.31), (4.2.40)]{MaHol}, we have
		 \begin{prop}\label{prop_cinf_prol}
		 	The function $t \in ]0,1] \to \exp(-u  L_{2,x}^t)(0,0)$ extends smoothly to $[0,1]$ by taking the value $\exp(-u  L_{2,x}^0)(0,0)$ at $t = 0$. All its derivatives are uniformly bounded on $x \in M$ and $u$, varying in a compact subset of $]0, + \infty[$.
		 	Moreover,
		 	\begin{equation}\label{eqn-der_vanishes}
		 		 \frac{\partial^{2i+1}}{\partial t^{2i+1}} \exp ( -u L_{2, x}^{t} ) (0,0)|_{t=0} = 0
		 	\end{equation}
		 \end{prop}
		 By Lemma \ref{lem-lapl_L_p-kernels}, Proposition \ref{prop_cinf_prol} and (\ref{eqn-L_p_L^t_idty}), we see that in Theorem \ref{thm-gen_asympt_exp} we have
		 \begin{equation}\label{eqn_a_i_u_as_derivative}
			 a_{k,u}(x) = \frac{1}{(2k)!} \frac{\partial^{2k}}{\partial t^{2k}} \exp (- u L^t_{2,x} / 2 )(0,0) |_{t=0}. 
		\end{equation}
		
		 \begin{thm}\label{eqn-L^t-general_est}
		 	For $t \in [0,1]$, there are sections $B_{t,r} \in \oplus_{j \geq 0} \cinf{M, \enmr{\Lambda^{j} (T^{*(0,1)}M) \otimes E}}$, $r \in \integ, r \geq -n$, such that for any $ k, m \in \nat, u_0 > 0$ there is $C > 0$ such that for any $ u \in ]0,u_0]$
			\begin{equation}\label{eq_L_t_gen_1}
		 		 \Big| \exp ( - u L_{2,x}^{t} / 2 )(0,0) - \sum_{r = -n}^{k}B_{t,r}(x) u^r\Big|_{\mathscr{C}^{m}(M \times [0, \, t_0])}
		 		\leq C u^{k+1},
		 	\end{equation}
			where the second coordinate of $M \times [0, \, t_0]$ represents $t$. Moreover,
		 	\begin{equation}\label{eq_L_t_gen_2}
		 		 \frac{\partial^{2i+1}}{\partial t^{2i+1}} B_{t,r}(x)|_{t=0}  = 0.
		 	\end{equation}
		 \end{thm}
		 \begin{proof}
		 	The proof of (\ref{eq_L_t_gen_1}) is done in \cite[(5.5.91)]{MaHol}. By  (\ref{eqn-der_vanishes}) and (\ref{eq_L_t_gen_1}), we get (\ref{eq_L_t_gen_2}).
		 \end{proof}
		 By Theorem \ref{thm_spec_gap_L_p} and (\ref{defn-L_2^t}) there are $t_0, \mu > 0$ such that for $t \in [0, t_0]$, we have
		 \begin{equation}\label{eqn_L_t_spec_gap}
		 	\spec(L^{t}_{2, x}) \subset \{ 0 \} \cup  [\mu, + \infty[, \qquad \ker (L^t_{2,x}) \subset \dfor[(0,0)]{X_0, E_0}.
		 \end{equation}
		 We fix $t_0$, which satisfies (\ref{eqn_L_t_spec_gap}). From now on, we only work with $t < t_0$.
		 \begin{sloppypar}
		 Recall that $L_{2,x}^{t}$ preserves $\integ$-grading on $\dfor[(0, \bullet)]{X_0, E_0}$.
		 We denote by $L_{2,x}^{t, >0}$ the restriction of the operator $L_{2,x}^{t}$ on the positive degree.
		 From (\ref{eqn_L_t_spec_gap}), 
		 \begin{equation}\label{eqn_pos_deg_fu}
		 	\exp(- u L_{2,x}^{t, >0})  =  F_u(L_{2,x}^{t})|_{\dfor[(0, >0)]{X_0, E_0}},
		 \end{equation}
		 where $F_u(L_{2,x}^{t})|_{\dfor[(0, >0)]{X_0, E_0}}$ is the restriction on positive degree terms of the operator $F_u(L_{2,x}^{t})$ defined in \cite[(4.2.21), (4.2.22)]{MaHol} .
		 We denote by $\exp(- u L_{2,x}^{t, >0})(Z, Z'); Z, Z' \in X_0$ the smooth kernel of the heat operator $\exp(- u L_{2,x}^{t, >0})$ with respect to the volume form $\, d v_{T_{x}M}$ on $X_0$. Then $\exp(- u L_{2,x}^{t, >0})(Z,Z')$ is the restriction of $\exp(- u L_{2,x})(Z,Z')$ on positive degree.
		 \end{sloppypar}
		 \begin{thm}\label{thm-exp_proj_expansion}
		 	For any $u_0 > 0, m \in \nat$, there are constants $c, C > 0$, such that for $u > u_0$
		 		\begin{equation}\label{thm_proj_est_1}
		 			 \big|\exp ( - u L_{2,x}^{t,>0} / 2 )(0,0)\big|_{\mathscr{C}^{m}(M \times [0, \, t_0])} \leq C\exp(-cu),
		 		\end{equation}
				where second coordinate of $M \times [0, \, t_0]$ represents $t$. Moreover, we have
				\begin{equation}\label{eqn_NF_r}
					a_{k,u}^{[>0]}(x) = \frac{1}{(2k)!}\frac{\partial^{2k}}{\partial t^{2k}} \exp ( - u L_{2,x}^{t,>0} / 2 )(0,0)|_{t=0}, \qquad \frac{\partial^{2k+1}}{\partial t^{2k+1}} \exp ( - u L_{2,x}^{t,>0} / 2 )(0,0)|_{t=0} = 0,
				\end{equation}
				where $[>0]$ means the projection onto positive degree terms.
		 \end{thm}
		 \begin{proof}
		 	Estimation (\ref{thm_proj_est_1}) is a consequence of {\cite[Corollary 4.2.6, (4.2.31)]{MaHol}} and (\ref{eqn_pos_deg_fu}). 
		 	The identities (\ref{eqn_NF_r}) follow directly from (\ref{eqn-der_vanishes}), (\ref{eqn_a_i_u_as_derivative}), (\ref{thm_proj_est_1}) and the discussion before Theorem \ref{thm-exp_proj_expansion}.
		 \end{proof}
		 
\subsection{Off-diagonal estimations of the heat kernel and related quantities}\label{paragr_off_diag}
	In this section we explain some results concerning off-diagonal expansion of the heat kernel of $L_{2,x}^{t}$. We don't claim originality on those results, as some of them already appeared in \cite[ \S 4.2,  \S 5.5]{MaHol} and some were implicit. This section is only used in the proof of Theorem \ref{thm-gen_asympt_orbi} in Section \ref{sect_orbi}. We fix $t_0$ as in (\ref{eqn_L_t_spec_gap}). We have the following off-diagonal version of Theorem \ref{thm-exp_proj_expansion}.
	\begin{thm}\label{thm_L_2_t_off_diag_u_infty}
		For any $m\in \nat, u_0 > 0 $ there are $ c, C, C' > 0$ such that for any $x \in M, u \geq u_0, Z,Z' \in T_xM$, we have the following inequality
		\begin{equation}\label{thm_L_2_t_off_diag_u_infty_eqn}
			\big| \exp(- u L_{2,x}^{t,>0})(Z, Z') \big|_{\mathscr{C}^{m}(M \times [0,\, t_0])} \leq C (1 + |Z| + |Z'|)^{C'} \exp ( -cu - c |Z - Z'|^2 / u ),
		\end{equation}
		where the second coordinate of $M \times [0,\, t_0]$ represents $t$.
	\end{thm}
	\begin{proof}
		By \cite[Theorem 4.2.5]{MaHol}, we get for some $ c, C, C' > 0$
		\begin{equation}\label{thm_L_2_t_off_diag_u_infty_eqn1}
			\big| \exp(- u L_{2,x}^{t})(Z, Z')  \big|_{\mathscr{C}^{m}(M \times [0, \, t_0])} \leq C (1 + |Z| + |Z'|)^{C'} \exp ( cu - c |Z - Z'|^2 / u ).
		\end{equation}
		By \cite[Corollary 4.2.6, (4.2.31)]{MaHol} and (\ref{eqn_pos_deg_fu}), we get 
		\begin{equation}\label{thm_L_2_t_off_diag_u_infty_eqn2}
			\big| \exp(- u L_{2,x}^{t,>0})(Z, Z')  \big|_{\mathscr{C}^{m}(M \times [0, \, t_0])} \leq C (1 + |Z| + |Z'|)^{C'} \exp ( -cu - c |Z - Z'| ),
		\end{equation}
		for some $ c, C, C' > 0$.
		We multiply (\ref{thm_L_2_t_off_diag_u_infty_eqn1}) and (\ref{thm_L_2_t_off_diag_u_infty_eqn2}) with suitable powers to get (\ref{thm_L_2_t_off_diag_u_infty_eqn}).
	\end{proof}
	
	\begin{thm}\label{thm_L_2_t_off_diag}
		For any  $ m \in \nat, u_0, c_0 > 0 $ there are  $c, C, C' > 0$ such that for $x \in M, u \in ]0, u_0]$ and $Z,Z' \in T_xM, |Z - Z'| \geq c_0$ we have the following inequality
		\begin{equation}
			\big| \exp(- u L_{2,x}^{t})(Z, Z') \big|_{\mathscr{C}^{m}(M \times [0,\, t_0])} \leq C (1 + |Z| + |Z'|)^{C'} \exp ( - c |Z - Z'|^2 / u ),
		\end{equation}
		where the second coordinate of $M \times [0,\, t_0]$ represents $t$.
	\end{thm}
	\begin{proof}
		The proof of this theorem proceeds exactly as in \cite[Theorem 4.2.5]{MaHol} with only one modification. One has to change the condition for \cite[(4.2.12)]{MaHol} by the following one: for any $m, m' \in \nat, c, c_0 > 0$ there are $c', C, C' > 0$ such that for $ u \in ]0, u_0], h \geq c_0, a \in \comp, |\Im a| \leq c' $ the following inequality holds (cf. \cite[(4.2.12)]{MaHol})
		\begin{equation}
			|a|^m |K_{u,h}^{(m')}(a)| \leq C \exp ( - c' h^2 / u),
		\end{equation}				
		where $K_{u,h}$ is defined in \cite[(4.2.11)]{MaHol}.
		 We leave the details to the reader.
	\end{proof}
	We denote $v := \sqrt{u}$. Let $\Delta^{T_xM}$ be the Bochner Laplacian on $(T_xM, g^{TM}_x)$. We set
	\begin{equation}
		L_{3,x}^{t} := \rho(|Z|/\epsilon) L_{2,x}^{t} + (1 - \rho(|Z|/\epsilon)) \Delta^{T_xM},
	\end{equation}
	with $\rho$ as in (\ref{defn_rho_fun}) and $\epsilon > 0$ satisfying (\ref{eqn_inj_rad}), (\ref{curv_positive}).
	We denote
	\begin{equation}\label{defn_L_4_t}
		L_{4,x}^{t,v} := S_v^{-1} u L_{3,x}^{t}S_v, \quad \text{with} \quad v = \sqrt{u},
	\end{equation}
	and $S_v$ is as in (\ref{defn-L_2^t}).
	We introduce the Sobolev norms
	\begin{align}\label{sob_norm_defn_tv}
		&  \norm{s}_{t, v, 0}^{2} := \int_{X_0} \norm{s(Z)}_{h}^{2} dv_{T_xM}(Z),\\
		&  \norm{s}_{t, v, m}^{2} :=  \sum_{k = 0}^{m} \sum_{i_1, \ldots, l_k = 1}^{2n} \norm{\nabla_{e_{i_1}} \ldots \nabla_{e_{i_k}}s }_{t,0}^{2},
	\end{align}
	where $s \in \cinf{X_0, \Lambda (T_Z^{*(0,1)}X_0) \otimes E_0}$, $\norm{\cdot}_{h}$ is the pointwise norm induced by $g^{TM}_{x}, h^{E}_{x}$, $\nabla$ is a usual derivative
	and $e_1, \ldots, e_{2n}$ are as in (\ref{defn_c_op}). We denote by $\textbf{H}_{t,v}^{m}, m \in \nat$ the Sobolev spaces induced by those norms. 
	\par 
	Then, similarly to \cite[Theorem 11.26]{BisLeb91}, \cite[Theorem 1.6.7]{MaHol}, there are $c_1, c_2 > 0$ such that for $t \in ]0,1], v \in ]0,1]$, we have the following estimations
	\begin{equation}
		\Re \scal{L_{4,x}^{t,v} s}{s}_{t,v,0} \geq c_1 \norm{s}_{t,v,1}^2 - c_2 \norm{s}_{t,v,0}^2,
	\end{equation}
	for $s \in \cinf{X_0, \Lambda (T_Z^{*(0,1)}X_0) \otimes E_0}$ of compact support.
	\par Then, similarly to \cite[Theorem 1.6.8]{MaHol}, for any $\lambda \in \comp$ as in \cite[Figure 1.1]{MaHol}, the inverse operator $(\lambda - L_{4,x}^{t,v})^{-1}$ is bounded as an operator operator on $\textbf{H}_{t,v}^{0}$. Then one can define the heat operator $\exp(-w L_{4,x}^{t,v}), w > 0$ by the integration over a contour of $(\lambda - L_{4,x}^{t,v})^{-1}$ as it was done in \cite[(1.6.48)]{MaHol}.
	\par Similarly, we define the heat operator $\exp(-w L_{3,x}^{t}), w > 0$. Even though the operators $L_{3,x}^{t}, L_{4,x}^{t,v}$ are not self-adjoint, by \cite[(1.6.31)]{MaHol}, their adjoints are of the same form as the operators themselves. Thus, all the arguments on the estimation of the kernels of $\exp(-w L_{3,x}^{t}), \exp(-w L_{4,x}^{t,v})$ can be repeated line in line from \cite{MaHol}.
	\par Now, similarly to (\ref{eqn-L_p_L^t_idty}), we have 
	\begin{equation}\label{eqn_L_3_L_4}
		\exp(- u L_{3,x}^{t})(Z, Z') = u^{-n} \exp(- L_{4,x}^{t,v}) (Z/v, Z'/v),
	\end{equation}
	where we denote by $\exp(- L_{3,x}^{t})(Z, Z'), \exp(- L_{4,x}^{t,v})(Z, Z')$ the smooth kernels of the heat operators $\exp(- L_{3,x}^{t}), \exp(- L_{4,x}^{t})$ corresponding to the volume form $dv_{T_xM}$.
	We have the following analogue of Lemma \ref{lem-lapl_L_p-kernels}, which follows from \cite[(5.5.81)]{MaHol} and (\ref{eqn_L_3_L_4})
	\begin{prop}\label{prop_l_2_l_4}
		There exists $u_0 > 0$ such that for any $m \in \nat$, there are $c, C > 0$ such that for any $u > u_0, Z,Z' \in T_xM,$ 
		$ |Z|, |Z'| < \epsilon$, we have
		\begin{equation}
			\big| \exp(-u L_{2, x}^{t})(Z, Z') - u^{-n} \exp(- L_{4,x}^{t,v}) (Z / v, Z' / v ) \big|_{\mathscr{C}^{m}(M \times [0, \, t_0])} \leq C \exp ( - c / u ),
		\end{equation}
		where the second coordinate of $M \times [0,\, t_0]$ represents $t$.
	\end{prop}
	\begin{prop}\label{prop_l_4_off_diag}
		For any $v_0 > 0, m \in \nat$ there are $c, C, C' > 0$ such that for $Z, Z' \in T_xM$, we have
		\begin{equation}
			\big| \exp(- L_{4,x}^{t,v})(Z, Z') \big|_{\mathscr{C}^{m}(M \times [0, \, t_0] \times [0, \, v_0])} \leq C (1 + |Z| + |Z'|)^{C'} \exp ( -c |Z -Z'|^2 ),
		\end{equation}
		where the second and third coordinates of $M \times  [0, \, t_0]  \times [0,v_0]$ represents $t$ and $v$ respectively.
	\end{prop}
	\begin{proof}
		When we fix $v = 1$, this proposition is a special case of \cite[Theorem 4.2.5]{MaHol}. In general, since the operator $L_{4,x}^{t,v}$  depends smoothly on $(t, v)$, we may repeat the argument of the proof of \cite[Theorem 4.2.5]{MaHol} as if the parameter $t$ in that theorem had two components $(t,v)$.
	\end{proof}
	\subsection{Proof of Propositions \ref{exp_a_i_u}, \ref{prop_b_p_i_exp}, \ref{prop-streng_zero}, \ref{prop-streng_infty}}\label{subsec_proof_aux}
	Here we finally prove Propositions \ref{exp_a_i_u}, \ref{prop_b_p_i_exp}, \ref{prop-streng_zero}, \ref{prop-streng_infty}; thus, completing the proof of Theorem \ref{thm-gen_asympt_zet_gen}. Then we also explain Remark \ref{rem_family}.
	\begin{proof}[Proof of Proposition \ref{exp_a_i_u}]
		From Theorem \ref{eqn-L^t-general_est} and (\ref{eqn_a_i_u_as_derivative}), we get (\ref{exp_a_i_u_zero}) with
		\begin{equation}\label{a_i_j_B_t_j}
			 a_k^{[j]}(x) = \frac{1}{(2k)!} \frac{\partial^{2k}}{\partial t^{2k}} B_{t,j}(x)|_{t = 0}.
		\end{equation}
		Now (\ref{exp_a_i_u_infty}) follows from Theorem \ref{thm-exp_proj_expansion}.
	\end{proof}
	\begin{proof}[Proof of Proposition \ref{prop_b_p_i_exp}]
		Firstly, we make a connection between $b_{p,i}$ and $B_{t,i}, t = \tfrac{1}{\sqrt{p}}$, defined in Theorem \ref{eqn-L^t-general_est}.
		By Theorem \ref{eqn-L^t-general_est}, Lemma \ref{lem-lapl_L_p-kernels} and (\ref{eqn-L_p_L^t_idty})  we see that there is $l \in \nat$ such that for any $k \in \nat$, there exist $c, C,C' > 0$ such that for any $p \in \nat^*, u \in ]0,1]$, we have
		\begin{equation}
			\textstyle \big| p^{-n} \exp (  - u \laplcomp_p / p )(x,x) - \sum_{r = -n}^{k}B_{t,r}(x) u^r \big|
		 \leq C u^{k+1} + Cp^l \exp ( - cp/u ) \leq C' u^{k+1},		
		\end{equation}
		thus, by (\ref{defn-b_pj}), we have
		\begin{equation}\label{eqn_b_p_i_B_t_i}
			\textstyle b_{p,i} = \int_M \str{N B_{t,i}(x)} \, d v_M(x), \quad t = 1/\sqrt{p}.
		\end{equation}
		From (\ref{eq_L_t_gen_2}), we get  the estimation from (\ref{eqn_b_i^j_as_der}) with
		\begin{equation}\label{eqn_b_i^j_B_t_i}
			 b_i^{[j]} = \frac{1}{(2j)!} \frac{\partial^{2j}}{\partial t^{2j}} \Big( \int_M \str{N B_{t,i}(x)} \, d v_M(x) \Big)|_{t=0}.
		\end{equation}
		Finally, (\ref{eqn_b_i^j_as_der}) follows from (\ref{a_i_j_B_t_j}) and (\ref{eqn_b_i^j_B_t_i}).
	\end{proof}	
	\begin{proof}[Proof of Proposition \ref{prop-streng_zero}]
		By Theorem \ref{eqn-L^t-general_est} and (\ref{eqn-der_vanishes}) we see that for any $k \in \nat; u_0, t_0 > 0$ there exists $C > 0$ such that for any $u \in ]0,u_0], t \in  ]0,t_0]$, we have
			\begin{multline}\label{eqn-first_der_exp_zero}
		 			 \Big|
		 				\frac{1}{t^{2k}} 
		 				\Big[
		 					\exp ( - u L_{2,x}^{t} / 2 )(0,0) - 
		 					 \sum_{r = -n}^{0} u^r B_{t,r}(x) 
		 		 			\\
		 					- 
		 					\sum_{i = 0}^{k-1} \frac{t^{2i}}{(2i)!}
		 					\frac{\partial^{2 i}}{\partial t^{2 i}} 
		 					\Big(
		 						\exp ( - u L_{2, x}^{t} / 2 ) (0,0) - 
		 						\sum_{r = -n}^{0} u^r
		 						B_{t,r}(x)
		 					\Big)|_{t=0}
		 				\Big]
		 			\Big|
		 			\leq C u,
		 		\end{multline}
		 for any $x \in M$. We conclude by Lemma \ref{lem-lapl_L_p-kernels}, (\ref{eqn-L_p_L^t_idty}), (\ref{eqn_a_i_u_as_derivative}), (\ref{eqn_b_p_i_B_t_i}) and (\ref{eqn_b_i^j_B_t_i}).
	\end{proof}
	\begin{proof}[Proof of Proposition \ref{prop-streng_infty}]
		We distinguish $2$ cases:
		\par 1. $u > \sqrt{p}$. In this case we proceed similarly to \cite[Theorem 5.5.11]{MaHol}. Theorem \ref{thm-lapl_spectral_gap} implies the first inequality in the following series of estimations and the second one is true by Theorem \ref{thm-gen_asympt_exp}
				\begin{align} \label{man_lrg_time_1}
					 p^{-n + k} \tr{ \exp ( - u \laplcomp_p^{(>0)} / p )} 
					&  \leq p^{-n} \tr{ \exp ( - \laplcomp_p^{(>0)} / p )} p^k  \exp ( - \tfrac{u-1}{p} cp )   \\ \nonumber
					&   \leq C' p^k \exp ( - cu/2 ) \exp ( - cu/2 )  
					\leq  C'' \exp ( - cu/2 ) .
				\end{align}
				Now, by Proposition \ref{exp_a_i_u} we obtain the following estimate for some $c_i, d_i, d', d > 0$ and any $x \in M$
				\begin{align}\label{man_lrg_time_2}
					\textstyle p^{k-j} \big| \str{N a_{j,u}(x)} \big| 
				&\leq c_j p^{k-j} \exp(- d_j u)  \leq d' \exp ( - du/2 ).
				\end{align}
				Then (\ref{prop-streng_infty_eqn}) follows from (\ref{man_lrg_time_1}) and (\ref{man_lrg_time_2}).
		\par 2. $u \leq \sqrt{p}$. This case is subtler. Let $t = 1/\sqrt{p}$. By Theorem \ref{thm-exp_proj_expansion}, we have
			\begin{multline}\label{eqn_exp_aux_zero_1}
				\textstyle p^k \Big| \int_M \str{N \exp ( - u L_{2,x}^{t} / 2 )(0,0)} \, dv_M(x)
			 	\\ \textstyle - \sum_{i=0}^{k} p^{-i} \int_M \str{N a_{i,u}(x)} \, dv_M(x) \Big| \leq C\exp(-cu).
			\end{multline}
			We conclude by Lemma \ref{lem-lapl_L_p-kernels}, (\ref{eqn-L_p_L^t_idty}), (\ref{eqn_exp_aux_zero_1}) and inequality $e^{-c p/u} \leq e^{-c \sqrt{p}/2 } e^{-cu/2}$.	 
	\end{proof}
	\begin{proof}[Proof of Remark \ref{rem_family}.]
Here we prove that the calculation of the asymptotics of the analytic torsion in Theorem \ref{thm-gen_asympt_zet} commutes with derivatives over the base in a family of manifolds.
	\par More precisely, let $\pi : X \to B$ be a proper holomorphic submersion of complex manifolds. We note by $T \pi$ the relative tangent bundle. Let $L$, $E$ be respectively a holomorphic line and vector bundles over $X$. We endow $L$, $E$ with Hermitian metrics $h^L$, $h^E$, and suppose that the metric $h^L$ is positive along the fibers. 
	We endow the fibers $M_s := \pi^{-1}(s), s \in B$ with a Kähler metric $g^{TM}_{s}$, which is smooth in $s \in B$.
Let's denote by $T(g^{TM}_{s}, h^{L^p \otimes E}|_{M_s})$ for $p \in \nat$, the analytic torsion of $L^p \otimes E|_{M_s}$ associated with $g^{TM}_{s}, h^{L}|_{M_s}, h^{E}|_{M_s}$. Then by Theorem \ref{thm-gen_asympt_zet}, for any $s \in B$, there are local coefficients $\alpha_i(s), \beta_i(s) \in \real, i \in \nat$ such that for any $k \in \nat$, as $p \to + \infty$, we have
	\begin{equation}
		-2 \log T(g^{TM}_{s}, h^{L^p \otimes E}|_{M_s}) = \textstyle \sum_{i = 0}^{k} p^{n-i} \big( \alpha_i(s) \log p + \beta_i(s) \big) + o(p^{n-k}).
	\end{equation}
First of all, from \cite[Theorem 1.3]{BGS3}, for any compact $K \subset B$, there is $p_0$ such that for $p \geq p_0$, the function $\log T(g^{TM}_{s}, h^{L^p \otimes E}|_{M_s})$ is smooth over $K$, for $p \geq p_0$. We will explain that the functions $\alpha_i(s), \beta_i(s)$ are also   smooth in $s \in B$ and for any $l \in \nat$, we have
	\begin{equation} 
		\Big\lVert -2 \log T(g^{TM}_{s}, h^{L^p \otimes E}|_{M_s}) - \sum_{i = 0}^{k} p^{n-i} \Big(  \alpha_i(s) \log p +  \beta_i(s) \Big)\Big\rVert_{\mathscr{C}^l(K)} 
		\leq c p^{n-k},
	\end{equation}
	for some $c > 0$.
From the proof of Theorem \ref{thm-gen_asympt_zet}, we see that it is enough to explain why Theorem \ref{thm-gen_asympt_exp} and Propositions  \ref{exp_a_i_u}, \ref{prop_b_p_i_exp},  \ref{prop-streng_zero}, \ref{prop-streng_infty} hold uniformly in $\mathscr{C}^{k}(\pi^{-1}(K))$, for any $k \in \nat$.
For brevity, we prove only the extension of Theorem \ref{thm-gen_asympt_exp}, as other extensions are done in a similar way.
For $s \in B$, we denote by $\laplcomp_{p, s}$ the Kodaira Laplacian on $M_s$ associated with $(L^p \otimes E) |_{M_s}$.
We need to prove that 
there are smooth sections $a_{i,u}(x)$, $i \in \nat$ of $\oplus_{j \geq 0}\enmr{\Lambda^{j}(T^{*(0, 1)} \pi)  \otimes E}$ over $X$, such that for any $l \in \nat$, $u > 0$, there is $c > 0$ such that for any $p \in \nat^*$, we have
		\begin{equation}\label{eq_1}
			\Big\lVert \exp ( - u \laplcomp_{p, \pi(x)}/p )(x,x) - \sum_{i = 0}^{k} a_{i,u}(x) p^{n-i} \Big\rVert_{\mathscr{C}^l(\pi^{-1}(K))}
			\leq c p^{n-k-1}.
		\end{equation}
		But to do so, essentially, we have to repeat the proof of \cite[Theorem 4.2.5]{MaHol} with practically no change, since $x$, varying in the fiber, is already treated as a parameter in it. We need only to replace the words “uniformly on $x \in M_s$" by “uniformly on $x \in \pi^{-1}(K)$".
	\end{proof}

\section{Proof of Theorem \ref{dzeta_asympt}}
	In this section we calculate the coefficients $\alpha_1, \beta_1$ from Theorem \ref{thm-gen_asympt_zet}. More precisely, in Section 4.1 we fix the notation and we derive the formal expressions for $\alpha_1, \beta_1$ in terms of $a_{1,u}$. In Section 4.2 we prove Theorem \ref{dzeta_asympt}. For this we express $\mathscr{O}_1, \mathscr{O}_2$ in terms of creation and annihilation operators and we use the Duhamel's formula for the derivative of the heat kernel to calculate explicitly $A(u)$. This is the most technical part of the article. In Section 4.3 we verify Theorem \ref{dzeta_asympt} on the projective line, we describe how Theorem \ref{dzeta_asympt} is related to arithmetic Riemann-Roch theorem \cite{Soule92} and we make a connection between Theorem \ref{dzeta_asympt} and a result from the article \cite[\S 4]{KlMa} by Klevtsov-Ma-Marinescu-Wiegmann.
	\subsection{Formal expressions for $\alpha_1, \beta_1$}
	Recall that $(M, g^{TM}, \Theta)$ is a compact Kähler manifold of complex dimension $n$ and $(E, h^E), (L, h^L)$ are holomorphic Hermitian vector bundles over $M$. We suppose
	\begin{equation}\label{suppos_prequant}
		\Theta = \omega = \frac{\imun}{2 \pi} R^L.
	\end{equation}
	We take $x \in M$. For the calculation we use the localization procedure from \cite[\S 1.6.2]{MaHol}, where authors use the normal coordinates instead of holomorphic.
	We do so since some part of the calculation was done in this context before. The formula (\ref{eqn_a_i_u_as_derivative}), which is the only prerequisite we need from Section \ref{sect_prelim}, still holds for this localization, since it relies on the wave-propagation technique (see \cite[Theorem 4.2.3]{MaHol}). 
	In this section every notation from Section 3 should be thought in the realms of the localization procedure from \cite[\S 1.6.2]{MaHol}.
	\par For the sake of convenience, in this section we use the following notation 
	\begin{equation}\label{not_A(u)}
		\begin{aligned}
			& A(u) = \int_M \str{N a_{1,u}(x)} \, dv_M(x), \\
			&  R(u) = C_{u} \int_M {\rm{Tr}_s}\Big[ N  e^{- 2 \pi u N}  {\rm{Id}}_{\Lambda^{\bullet} (T^{*(0,1)}_{x}M) \otimes E_x} \Big] \, dv_M(x).
		\end{aligned}
	\end{equation}
	\begin{prop}\label{prop_strace_refin}
		For any $u > 0$, we have
		\begin{equation}
			\textstyle \lim_{p \to \infty} p \big( p^{-n}\str{N \exp ( - u \laplcomp_p / p ) } - R(u) \big) = A(u),
		\end{equation}
		and the convergence is uniform as $u$ varies in a compact subset of $]0, +\infty[$. 
	\end{prop}
	\begin{proof}
		It follows from definitions of $A(u), R(u)$ and Theorems \ref{thm-gen_asympt_exp}, \ref{thm-HK_form}.
	\end{proof}
	Now by (\ref{thm-a_i-b_i_formula}), we have the following identities (see Notation \ref{not_coefficients})
	\begin{equation}\label{eqn_alph_1}
		\alpha_1 = A^{[0]}, \qquad \qquad \beta_1 = - M[A]'(0).
	\end{equation}
	
\subsection{Proof of Theorem \ref{dzeta_asympt}}
	In this section we prove Theorem \ref{dzeta_asympt}. For this we give an explicit formula for $a_{1,u}(x), A(u)$ and then plug it in (\ref{eqn_alph_1}).
	\par Let $w_1, \ldots, w_n$ be an orthonormal basis of $(T_x^{(1,0)}M, h^{T^{(1,0)}_{x}M})$ and let $w^1, \ldots, w^n$ be its dual basis. For $j=1, \ldots, n,$ the vectors $e_{2j-1} = \tfrac{1}{\sqrt{2}}(w_j + \overline{w}_j), e_{2j} = \tfrac{\imun}{\sqrt{2}}(w_j - \overline{w}_j)$ form an orthonormal basis of $T_xM$. This basis identifies $T_xM$ and $\real^{2n}$. Let's introduce the complex coordinates $(z_1, \ldots, z_n)$ on $\comp^n \simeq \real^{2n}$ such that $Z = z + \overline{z}$ and $w_j = \sqrt{2} \tfrac{\partial}{\partial z_j}, \overline{w}_j = \sqrt{2} \tfrac{\partial}{\partial \overline{z}_i}$. We may consider $z, \overline{z}$ as vector fields by identifying $z$ to $\sum_i z_i  \tfrac{\partial}{\partial z_j}$ and $\overline{z}$ to $\sum_i \overline{z}_i  \tfrac{\partial}{\partial \overline{z}_i}$,
	\par Now we define \textit{creation} and \textit{annihilation} operators (see (\ref{defn-L_2^0}), (\ref{suppos_prequant}))
		\begin{align}\label{defn_creat_annih}
			 b_j = -2 \nabla_{0, \tfrac{\partial}{\partial z_j}} 
			= - 2 \frac{\partial}{\partial z_j} + \pi \overline{z}_j,  \qquad
			b_j^{+} = 2 \nabla_{0, \tfrac{\partial}{\partial \overline{z}_j}} 
			= 2 \frac{\partial}{\partial \overline{z}_j} + \pi z_j.
		\end{align}
	 We recall that by $\scal{\cdot}{\cdot}$ we mean the $\comp$-bilinear extension of $g^{TM}$. From now on, we use Einstein summation convention.
	\begin{thm}\label{thm-O_2_form}
		The following identities hold 
		\begin{align}
				\textstyle L_{2, x}^{0} &=  \sum_j  b_j b_j^{+}  + 4 \pi N, \qquad \qquad
				\mathcal{O}_1 = 0, \label{form_o_1}   \\	 
				\label{form_o_2}
				\mathcal{O}_2 &= \frac{1}{3} \Big \langle R_x^{TM} \Big( \overline{z}, \frac{\partial}{\partial z_i}\Big) \overline{z}, \frac{\partial}{\partial z_j} \Big \rangle b_i^{+}b_j^{+} +
				\frac{1}{3} \Big \langle R_x^{TM}\Big(z, \frac{\partial}{\partial \overline{z}_i}\Big)z, \frac{\partial}{\partial \overline{z}_j} \Big \rangle   b_i b_j  \\
				 & \phantom{= \,\,} - \frac{1}{3}  \Big \langle R_x^{TM}\Big(z, \frac{\partial}{\partial \overline{z}_i}\Big) \overline{z}, \frac{\partial}{\partial z_j} \Big \rangle    b_i b_j^{+} -
				\frac{1}{3}  \Big \langle R_x^{TM}\Big(\overline{z}, \frac{\partial}{\partial z_i}\Big)z, \frac{\partial}{\partial \overline{z}_j} \Big \rangle  b_i^{+} b_j  \nonumber \\
				 & \phantom{= \,\,} - 2 R_x^{E}\Big(\frac{\partial}{\partial z_i}, \frac{\partial}{\partial \overline{z}_i}\Big) - \frac{r_x^{M}}{6} \nonumber \\
				 & \phantom{= \,\,} + \frac{2}{3}  \Big \langle R_x^{TM}\Big(\overline{z}, \frac{\partial}{\partial z_i}\Big) \frac{\partial}{\partial \overline{z}_i} \frac{\partial}{\partial z_j} \Big \rangle b_j^{+} - \frac{2}{3} \Big \langle R_x^{TM}\Big(z, \frac{\partial}{\partial \overline{z}_i}\Big) \frac{\partial}{\partial z_i}, \frac{\partial}{\partial \overline{z}_j} \Big \rangle b_j  \nonumber \\
				 & \phantom{= \,\,} + \frac{\pi}{3}  \Big \langle R_x^{TM}(z,\overline{z}) \overline{z}, \frac{\partial}{\partial z_i} \Big \rangle b_i^{+} - \frac{\pi}{3} \Big \langle  R_x^{TM}(z,\overline{z}) z, \frac{\partial}{\partial \overline{z}_i} \Big \rangle b_i  \nonumber \\
				 & \phantom{= \,\,} -  R_x^{E}\Big(\overline{z}, \frac{\partial}{\partial z_i}\Big) b_i^{+} + R_x^{E}\Big(z, \frac{\partial}{\partial \overline{z}_i}\Big) b_i  \nonumber \\
				 & \phantom{= \,\,} -  R_x^{\Lambda^{\bullet} (T^{*(0,1)}M)}\Big(\overline{z}, \frac{\partial}{\partial z_i}\Big) b_i^{+} + R_x^{\Lambda^{\bullet} (T^{*(0,1)}M)}\Big(z, \frac{\partial}{\partial \overline{z}_i}\Big) b_i  \nonumber \\
				& \phantom{= \,\,} + 2 R_x^{\det}\Big(\frac{\partial}{\partial z_i},\frac{\partial}{\partial z_j}\Big) \overline{w}^j \wedge i_{\overline{w}_i} + 4 R_x^{E}\Big(\frac{\partial}{\partial z_i}, \frac{\partial}{\partial z_j}\Big) \overline{w}^j \wedge i_{\overline{w}_i} \nonumber.
		\end{align}				
	\end{thm}		
	\begin{proof}
		In \cite[Theorem 5.1]{MaBerg06} (cf. \cite[Theorem 4.1.25]{MaHol}) authors obtained this result in degree $(0,0)$. In $\mathcal{O}_2$,  the last $2$ lines of its formula is the only contribution of non-zero degree. Theorem \ref{thm-O_2_form} was obtained in \cite[Theorem 2.2]{MaMar06} for $Spin^c$-Dirac operator.
	\end{proof}
	\par  From (\ref{defn-L_2^0}), (\ref{suppos_prequant}), (\ref{form_o_1}) and Mehler formula for harmonic oscillator (see \cite[Appendix E 2.2]{MaHol}), we get
	\begin{thm}\label{thm-HK_form}
		We have the following identity
		\begin{equation}
			\textstyle \exp(-u L_{2,x}^{0})(Z,0) = e^{-4 \pi u N} C_{2 u} \exp ( -B_{2 u} \norm{Z}^2 ) {\rm Id}_{ \Lambda^{\bullet} (T^{*(0,1)}_{x}M) \otimes E_x},
		\end{equation}
		where the operator $N$ is defined in (\ref{defn_number_oper}), $Z = (z_1, \ldots, z_n), \norm{Z}^2 = \sum |z_i|^2$ and
		\begin{equation}
			 B_{u} = \frac{\pi}{2 \tanh(\pi u)}, \qquad C_{u} = \frac{1}{(1 - e^{-2 \pi u})^n}.
		\end{equation}
	\end{thm}
	\begin{sloppypar}
	From Duhamel's formula (see \cite[Theorem 4.17]{MaBerg06}), (\ref{eqn_a_i_u_as_derivative}) and (\ref{form_o_1}), we get
		\begin{equation}\label{eqn_a_1_volterra}
			\textstyle a_{1,u}(x) = - \int_{0}^{u/2} \int_{X_0} e^{-vL_{2, x}^{0}}(0,Z) ( \mathcal{O}_2 e^{-(u/2 - v)L_{2, x}^{0}} )(Z, 0) d Z \, d v
		\end{equation}
	From (\ref{eqn_a_1_volterra}), to calculate $a_{1,u}$ we have to calculate $\mathcal{O}_2 e^{-u L_{2, x}^{t}}(Z,0)$ for $Z \in T_xM$.
	To simplify this calculation, we \textit{omit} the terms of the form $P(z_1, \overline{z}_1, z_2, \cdots, \overline{z}_n) \exp ( - v L_{2,x}^{0}/2 ) (Z,0)$, where $P$ is a monomial with different degrees of $z_i$ and $\overline{z}_i$ for some $i \in \nat^*, i \leq n$, since from Theorem \ref{thm-HK_form} those terms disappear after the integration in $Z$ in (\ref{eqn_a_1_volterra}). We denote by $\sim$ the identification up to such \textit{omission}.
		 We note
		 		\end{sloppypar}
		 \begin{equation}\label{notation_curv}
		 	\begin{aligned}
		 		&R_{i\overline{j}k\overline{l}} = \scal{R^{TM}_{x}(\tfrac{\partial}{\partial z_i}, \tfrac{\partial}{\partial \overline{z}_j}) \tfrac{\partial}{\partial z_k}}{\tfrac{\partial}{\partial \overline{z}_l}}, 
		 		&&R^E_{i \overline{j}} = R_x^{E}\big( \tfrac{\partial}{\partial z_i}, \tfrac{\partial}{\partial \overline{z}_j} \big), \\
				&R^{\Lambda}_{i \overline{j}}  = R_x^{\Lambda^{0, \bullet} (T^*M)}\big( \tfrac{\partial}{\partial z_i}, \tfrac{\partial}{\partial \overline{z}_j} \big),  
				&&R^{\det}_{i \overline{j}} = R_x^{\det}\big( \tfrac{\partial}{\partial z_i}, \tfrac{\partial}{\partial \overline{z}_j} \big),
		 	\end{aligned}
		 \end{equation}
		where $R^{\det}$ is the Chern curvature of $(\det T^{(1,0)} M,  h^{\det})$ for the induced by $h^{T^{(1,0)}M}$ Hermitian metric $h^{\det}$. 
		We constantly use the following well-known symmetries of the curvature tensor
		\begin{equation}
			R_{i\overline{i}j\overline{j}} = R_{i\overline{j}j\overline{i}} = R_{j\overline{j}i\overline{i}} = R_{j\overline{i}i\overline{j}},  \qquad R_{i\overline{j}k\overline{l}} = R_{k\overline{l} i\overline{j}}.
		\end{equation}

		\begin{lem}\label{lem_o_2_exp}
			For $u >0, Z \in T_xM$, we have
			\begin{multline}\label{eqn_aux_a_1_3} 
			( \mathcal{O}_2 e^{-u L_{2, x}^{t}} )(Z,0) \sim \Big[ \frac{2}{3}  R_{i\overline{i}j\overline{j}} |z_i|^2 |z_j|^2 ( 2\pi^2 - \delta_{ij} \pi^2 ) - \frac{4}{3}R_{i\overline{i}j\overline{j}}   - 2 R^{E}_{i \overline{i}} 
			 \\ 			
			+ 2 \pi R^{E}_{i \overline{i}} |z_i|^2  + 2 \pi R^{\Lambda}_{i \overline{i}}  |z_i|^2 + 2R^{\det}_{i \overline{j}} \overline{w}^j \wedge i_{\overline{w}_i} + 4R^{E}_{i \overline{j}} \overline{w}^j \wedge i_{\overline{w}_i}  \Big] e^{-u L_{2, x}^{t}}(Z,0) .
			\end{multline}
		\end{lem}		
		\begin{proof}
		From Theorem \ref{thm-O_2_form}, we get
		\begin{align}
		 ( \mathcal{O}_2 & e^{-u L_{2, x}^{t}} )(Z,0) \sim  \label{eqn_aux_a_1_1} \\ \nonumber
			& 
			\Big[ 
			\frac{1}{3}(2 - \delta_{ij}) R_{i\overline{i}j\overline{j}} \big(  \overline{z}_i \overline{z}_j b_i^{+} b_j^{+}
			+  z_i z_j b_i b_j \big) \\ \nonumber
			& \phantom{\Big\{ }  + \frac{1}{3} (1 - \delta_{ij}) R_{i\overline{i}j\overline{j}} \big(   z_i \overline{z}_j b_i b_j^{+}
			+  z_j \overline{z}_i  b_i^{+} b_j \big) \\ \nonumber
			& \phantom{\big\{ } + \frac{1}{3}  R_{i\overline{i}j\overline{j}} \big( z_j \overline{z}_j b_i b_i^{+} 
			+   z_j \overline{z}_j b_i^{+} b_i \big) 
			- 2 R^{E}_{i \overline{i}}  - \frac{1}{6} r_x^{M}   \\ \nonumber
			& \phantom{\big\{ } + \frac{2}{3} R_{i\overline{i}j\overline{j}}  \big(  \overline{z}_j b_j^{+} - z_j b_j \big)  
			 - \frac{\pi}{3} (2 - \delta_{ij}) R_{i\overline{i}j\overline{j}} z_j \overline{z}_j \big(   \overline{z}_i b_i^{+} +  z_i b_i \big)  \\ \nonumber
			& \phantom{\big\{ } + R^{E}_{i \overline{i}} \big(  \overline{z}_i b_i^{+} + z_i b_i \big)  
			+ R^{\Lambda}_{i \overline{i}}  \big( \overline{z}_i b_i^{+} +  z_i b_i \big) \\ \nonumber
			& \phantom{\big\{ }
			+ 2R^{\det}_{i \overline{j}} \overline{w}^j \wedge i_{\overline{w}_i} + 4 R^{E}_{i \overline{j}} \overline{w}^j \wedge i_{\overline{w}_i}  \Big] e^{-u L_{2, x}^{t}}(Z,0),
		\end{align}
		where $\delta_{ij}$ is the Kronecker delta.
		We have the following formulas from Theorem \ref{thm-HK_form}
		\begin{equation}\label{eqn_b_HK}
			\begin{aligned}
				&( b_i  e^{-u L_{2, x}^{t}} )(Z,0) = ( \pi + 2 B_{2u} ) \overline{z}_i e^{-u L_{2, x}^{t}} (Z, 0), \\
				&( b_i^{+}  e^{-u L_{2, x}^{t}} )(Z,0) = ( \pi - 2 B_{2u} ) z_i e^{-u L_{2, x}^{t}} (Z, 0). 
			\end{aligned}
		\end{equation}
		Let's recall the following identity
		\begin{equation}\label{eqn_r_M_(i,j)}
			\textstyle r^M_{x} = \sum_{i,j} \scal{R(e_i, e_j)e_i}{e_j} = 2 \textstyle \sum_{i,j} \scal{R(w_i, \overline{w}_i) w_j}{\overline{w}_j} = 8 \sum_{i,j} R_{i\overline{i}j\overline{j}}.
		\end{equation}
		From (\ref{eqn_aux_a_1_1}),  (\ref{eqn_b_HK}) and  (\ref{eqn_r_M_(i,j)}), we get (\ref{eqn_aux_a_1_3}).
		\end{proof}

	\begin{lem}\label{lem_a_1_form}
		For $u > 0$ and $x \in M$, we have
		\begin{align}
			a_{1,u}(x) = \bigg[ &- \frac{4}{3} R_{i\overline{i}j\overline{j}} (1 - e^{-2 \pi u})^{-2} \Big( \frac{u}{2} (1 + 4 e^{-2 \pi u} + e^{-4 \pi u}) - \frac{3}{4\pi} (1 - e^{-4 \pi u}) \Big) \label{eqn_a_1_u_explicit} \\ \nonumber
			& + \frac{4}{6}R_{i\overline{i}j\overline{j}} u + R^{E}_{i \overline{i}} u 
			 - 2 R^{E}_{i \overline{i}} (1 - e^{-2 \pi u})^{-1} \Big( \frac{u}{2} +  \frac{u}{2} e^{- 2 \pi u} - \frac{1}{2\pi} (1 - e^{-2 \pi u}) \Big) \\ \nonumber
			& - 2  R^{\Lambda}_{i \overline{i}}  (1 - e^{-2 \pi u})^{-1} \Big( \frac{u}{2} + \frac{u}{2} e^{- 2 \pi u} - \frac{1}{2\pi} (1 - e^{-2 \pi u}) \Big) \\ \nonumber
			& - \big(R^{\det}_{i \overline{j}} \overline{w}^j \wedge i_{\overline{w}_i} + 2 R^{E}_{i \overline{j}} \overline{w}^j \wedge i_{\overline{w}_i} \big) u 
			\bigg]			
			\frac{e^{-2 \pi u N}}{ (1 - e^{-2 \pi u})^{n}} .
		\end{align}
	\end{lem}
	\begin{proof}
		From Theorem \ref{thm-HK_form} and the fact that $\exp(-uL_{2,x}^{t}), u > 0$ is a semigroup, we get
		\begin{equation}\label{lem_int_calcul_1}
			 \int_{0}^{u} dv \int_{X_0} e^{-v L_{2, x}^{0} }(0,Z)  e^{-(u-v) L_{2, x}^{0} }(Z,0) d Z = u \frac{e^{-4 \pi u N}}{(1 - e^{-4 \pi u})^{n}}.
		\end{equation}
		Similarly, we get 
		\begin{align}\label{lem_int_calcul_2}
				&  \int_{0}^{u} dv \int_{X_0} e^{-v L_{2, x}^{0} }(0,Z) |z_j|^2 e^{-(u-v) L_{2, x}^{0} }(Z,0) d Z  \\
				& \nonumber  \phantom{\int_{0}^{u} dv \int e^{-v L_{2, x} }(0,Z)}
				= 
				\frac{e^{-4 \pi u N}}{\pi (1 - e^{-4 \pi u})^{n+1}} \Big( u + u e^{- 4 \pi u} - \frac{1}{2\pi} (1 - e^{-4 \pi u}) \Big), \nonumber \\	
				&  \int_{0}^{u} dv \int_{X_0} e^{-v L_{2, x}^{0} }(0,Z) |z_i|^2 |z_j|^2 e^{-(u-v) L_{2, x}^{0} }(Z,0) d Z   \\ \nonumber
				&  \phantom{\int_{0}^{u} dv \int e^{-v L_{2, x} }(0,Z)} = \frac{e^{-4 \pi u N} 2^{\delta_{ij}}}{\pi^2 (1 - e^{-4 \pi u})^{n+2}} \Big( u (1 + 4 e^{-4 \pi u} + e^{-8 \pi u}) - \frac{3}{4\pi} (1 - e^{-8 \pi u}) \Big).
			\end{align}
			We get (\ref{eqn_a_1_u_explicit}) from Lemma \ref{lem_o_2_exp}, (\ref{eqn_a_1_volterra}), (\ref{lem_int_calcul_1}) and (\ref{lem_int_calcul_2}).
	\end{proof}
		Now, we introduce the functions $g_1, g_2, \tilde{g}_2, \tilde{g}_3 : \real \to \real$ by
	\begin{equation}\label{defn_g_fun}
		\begin{aligned}
			&g_1(u) = \frac{e^{-2 \pi u}}{1 - e^{-2 \pi u}}, \qquad  &&g_2(u) = \frac{e^{-2 \pi u}}{(1 - e^{-2 \pi u})^2}, \\
			&\tilde{g}_2(u) = \frac{u e^{-2 \pi u}}{(1 - e^{-2 \pi u})^2},  \qquad  &&\tilde{g}_3(u) = \frac{u e^{-2 \pi u}}{(1 - e^{-2 \pi u})^3}.
		\end{aligned}
	\end{equation}
	
	\begin{lem}\label{lem_A_expl}
		For $u > 0$, we have
		\begin{multline}\label{eqn_A_expl}
			 A(u) = - \rk E \int_M c_1(TM) \frac{\omega^{n-1}}{(n-1)!} \left( g_2(u) + \tfrac{n}{2}g_1(u) - 2 \pi \tilde{g}_3(u) \right) \\ 
			-  \int_M c_1(E) \frac{\omega^{n-1}}{(n-1)!} \left( n g_1(u) - 2 \pi \tilde{g}_2(u) \right).
		\end{multline}
	\end{lem}
	\begin{proof}
		Let $a_{i,j} \in \enmr{E}; i,j = 1, \ldots, n$ then
		\begin{align}
			& \str{e^{-2 \pi u N}} = \rk E (1 - e^{-2 \pi u})^{n}, \\
			& \textstyle \str{\sum_{k, l=1}^{n} a_{k,l} \overline{w}^k \wedge i_{\overline{w}_l} e^{-2 \pi u N}} 
			= \sum_{j =1}^n (-1)^j e^{- 2 \pi j u} \sum_{k=1}^{n} {\rm{Tr}}^E [ a_{k,k} ]  \binom{n-1}{j-1} \\
			& \textstyle \phantom{\str{\sum_{k, l=1}^{n} a_{k,l} \overline{w}^k \wedge i_{\overline{w}_l} e^{-2 \pi u N}}} 
			= - {\rm{Tr}}^E \big[ \sum_{i = 1}^{n} a_{i,i} \big] e^{- 2 \pi u} (1 - e^{- 2 \pi u})^{n-1}. \nonumber
		\end{align}
		By taking derivatives of those identities, we get
		\begin{align}\label{lem_trs_comp}
				& \textstyle \str{N e^{-2 \pi u N}} = - \rk E n e^{-2 \pi u} (1 - e^{-2 \pi u})^{n-1}, \\
				& \textstyle \str{N \sum_{i, j=1}^{n} a_{i,j} \overline{w}^i \wedge i_{\overline{w}_j} e^{-2 \pi u N}} 
				= - {\rm{Tr}}^E \big[ \sum_{i = 1}^{n} a_{i,i} \big] e^{-2 \pi u} (1 - ne^{-2 \pi u})(1 - e^{- 2 \pi u})^{n-2}. \nonumber
		\end{align}
		By (\ref{notation_curv}), we have 
	\begin{equation}\label{eqn_R_lamb}
		\begin{aligned}
			\textstyle \sum_i R^{\Lambda}_{i \overline{i}}  =  \tfrac{1}{2} \scal{R^{T^{(1,0)}M}(w_k, \overline{w}_k) w_i}{\overline{w}_j} \overline{w}^j \wedge i_{\overline{w}_i} 
			= \sum_{i,j} R^{\det}_{i \overline{j}} \overline{w}^j \wedge i_{\overline{w}_i} .
		\end{aligned}	
	\end{equation}
	For a $2$-form $\alpha$, we define the function $\Lambda_{\omega} \left[ \alpha \right]$ by the identity $\Lambda_{\omega} \left[ \alpha \right] \tfrac{\omega^{n}}{n!} = \alpha \tfrac{\omega^{n-1}}{(n-1)!}$, 
	then
	\begin{equation}\label{r_ii_iden}
	\begin{aligned}
		& \textstyle r^M_{x} =  8 \sum_{i,j} R_{i\overline{i}j\overline{j}} = 4 \sum_i R^{\det}_{i \overline{i}} = 4 \pi \Lambda_{\omega} \big[ c_1(T^{(1,0)}M, h^{T^{(1,0)}M}) \big],\\
		& \textstyle \sum_i \tr{R^E_{i \overline{i}}} = \pi \Lambda_{\omega} \left[ c_1(E, h^E) \right].
	\end{aligned}
	\end{equation}
	By Lemma \ref{lem_a_1_form}, (\ref{eqn_r_M_(i,j)}), (\ref{lem_trs_comp}), (\ref{eqn_R_lamb}) and (\ref{r_ii_iden}) to get 
	\begin{multline}\label{eqn_str_a_1_g}
			 \str{N a_{1,u}(x)} = - \rk E \Lambda_{\omega} \big[ c_1(T^{(1,0)}M, h^{T^{(1,0)}M})  \big] \Big( g_2(u) + \tfrac{n}{2}g_1(u) - 2 \pi \tilde{g}_3(u) \Big) \\
			 - \Lambda_{\omega} \big[ c_1(E, h^E) \big]  \Big( n g_1(u) - 2 \pi \tilde{g}_2(u) \Big).
	\end{multline}
	By (\ref{not_A(u)}) and (\ref{eqn_str_a_1_g}) we deduce (\ref{eqn_A_expl}).	
	\end{proof}
\begin{proof}[Proof of Theorem \ref{dzeta_asympt}.]
	We verify that as $u \to 0$,
	\begin{equation}\label{lem_g_exp}
		\begin{aligned}
			&g_1(u) = g_1^{[-1]} u^{-1} - \tfrac{1}{2} + O(u),  \qquad 
			&& g_2(u) = g_2^{[-2]} u^{-2} + g_2^{[-1]} u^{-1} - \tfrac{1}{12} + O(u), \\
			&\tilde{g}_2(u) = \tilde{g}_2^{[-1]} u^{-1} + O(u),  \qquad  
			&& \tilde{g}_3(u) = \tilde{g}_3^{[-2]} u^{-2} + \tilde{g}_3^{[-1]} u^{-1} + O(u). 
		\end{aligned}
	\end{equation}
	From Lemma \ref{lem_A_expl}, (\ref{eqn_alph_1}) and (\ref{lem_g_exp}), we get (\ref{eqn_alpha_1}).
	\par Let $\zeta(z)$ be the Riemann zeta function.
	By (\ref{defn_g_fun}), we have
		\begin{align}\label{lem_g_mell_1}
			& \textstyle \mell{g_1}(z)
				= \frac{1}{\Gamma(z)} \int_0^{+ \infty} \sum_{j \geq 1} e^{-2 \pi j u} u^{z-1} \,du 
				= (2 \pi)^{-z} \zeta(z),
	\end{align}
	Similarly, we get
	\begin{equation}\label{lem_g_mell_2}
		\begin{aligned}
			&\mell{g_2}(z) = (2 \pi)^{-z} \zeta(z-1), \quad
			 \mell{\tilde{g}_2}(z) = z (2 \pi)^{-(z+1)} \zeta(z),  \\
			 &\mell{\tilde{g}_3}(z) = z (2 \pi)^{-(z+1)} \big( \zeta(z-1) + \zeta(z) \big)/2. 
		\end{aligned}	
	\end{equation}
	We recall that
	\begin{equation}\label{eqn_zeta_ident}
		\zeta'(0) = -\tfrac{1}{2} \log(2 \pi) , \qquad \zeta(0) = -\tfrac{1}{2}, \qquad \zeta(-1) = - \tfrac{1}{12}.
	\end{equation}
	From Lemma \ref{lem_A_expl}, (\ref{eqn_alph_1}), (\ref{lem_g_mell_1}), (\ref{lem_g_mell_2}) and (\ref{eqn_zeta_ident}), we get (\ref{eqn_beta_1}).
	\end{proof}
\subsection{Relations to previous works} \label{rel_prev_work}
	\paragraph{Verification.}\label{verification}
	The analytic torsion of $\mathbb{CP}^n$ for $n \geq 1$ with trivial line bundle was computed by Gillet-Soulé and Zagier in \cite{GilSoulTodd}  and it played an important role in the formulation and proof of the arithmetic Riemann-Roch theorem by Gillet-Soulé in \cite{GilSoul92}. Later, it was reobtained by Bost in \cite{BostCPn} as a direct consequence of Bismut-Lebeau immersion theorem \cite{BisLeb91}, \cite{BisKoz}. 
	\par 
	Let's denote now by $M = \mathbb{CP}^1$ and  by $L = \mathcal{O}(1)$ the hyperplane line bundle. We endow $\mathcal{O}(-1)$ with the Hermitian metric induced from the inclusion
	$\mathcal{O}(-1) \xhookrightarrow{} \comp^2 : ([z], \lambda z) \mapsto \lambda z, z \in \comp^2 \setminus \{0\}$.  Let $h^L$ be the dual Hermitian metric on $L = \mathcal{O}(-1)^*$.
	 The Fubiny-Study metric $g^{TM}$ on $TM$ is by definition the metric, associated to the positive 2-form $\omega = c_1(L,h^L) = \tfrac{1}{2 \pi} \imun R^L$. 
	 By \cite[Theorem 18]{Koh95} or \cite[(12), (73), (74)]{KlMa}\footnote{Notice the last two terms, which appear because the metric considered in the article \cite{KlMa} differs from our metric by a factor $2 \pi$.},  we get for $p \geq 1$
	\begin{multline}\label{form_zeta_asympt_proj}
		\textstyle  2 \log T(g^{TM}, h^{L^p}) = 2 \sum_{j = 1}^{p} (p - j) \log(j+1) - (p+1) \log(p+1)!  \\ 
		\textstyle - 4 \zeta'(-1) + \frac{1}{2}(p+1)^2  - \frac{\log( 2 \pi)}{2}p - \frac{2}{3} \log(2 \pi).
	\end{multline}
	By \cite[(5.11.1), (5.17.2), (5.17.5)]{NistFunctions}, we have the following asymptotic expansions of Barnes G-function and factorial as $p \to + \infty$
	\begin{equation}\label{eqn_NT_exp}
		\begin{aligned}
			& \textstyle \log  \prod_{i=1}^{p - 1} i!  = \frac{1}{2} p^2 \log p - \frac{3}{4} p^2 + \frac{1}{2} \log(2 \pi) p - \frac{1}{12} \log p + \zeta'(-1) + O(p^{-1}),  \\
			& \textstyle \log p! = p \log p - p + \frac{1}{2} \log p + \frac{1}{2} \log (2 \pi) + \frac{1}{12}p^{-1} + O(p^{-2}).
		\end{aligned}
	\end{equation}
	We note that from \cite{NistFunctions} we can actually get each coefficient in the expansion of 
	$\log T(g^{TM}, h^{L^p})$.
	We substitute (\ref{eqn_NT_exp}) into (\ref{form_zeta_asympt_proj}) and we get, as $p \to + \infty$
	\begin{equation}\label{asymp_proj}
		2 \log T(g^{TM}, h^{L^p}) = - \tfrac{1}{2}p \log p - \tfrac{2}{3} \log p - \tfrac{1}{6} \log (2 \pi) - \tfrac{7}{12} - 2 \zeta'(-1) + O(p^{-1}).
	\end{equation}
	Since $\int_M c_1(TM) = 2$, this formula coincides with Theorem \ref{dzeta_asympt} for $E$ is trivial.
	\par To check the coefficients of $c_1(E)$ for general $n$-dimensional manifold $M$, the reader may compare the coefficients of $k^{n-1} \log k$ and $k^{n-1}$ from Theorem \ref{dzeta_asympt} applied for $E = L, p = k-1$ and $E = \mathcal{O}_M, p = k$.
	\paragraph{Connection with the arithmetic Riemann-Roch theorem.}\label{ar_rrh} Now let's describe an informal connection between $\zeta'(-1)$, which appears in Theorem \ref{dzeta_asympt}, and the one which appears in the $R$-genus of Gillet-Soulé.
	\par Let $X$ be an arithmetic variety in the sense of the book of Soulé \cite[p. 55]{Soule92}. In \cite{GilSArithm}, Gillet-Soulé defined arithmetic Chow groups $\widehat{CH}^k(X)$, for $k \in \nat$. Those groups are generated by pairs $(Z, g_Z)$, where $Z$ is a cycle of codimension $k$ and $g_Z$ is a current over $X(\comp)$ of bi-degree $(k-1, k-1)$, for which $\tfrac{\dbar \partial}{2 \pi \imun} g_Z + \delta_Z$ is smooth. There is an intersection pairing $\widehat{CH}^r(X) \times \widehat{CH}^q(X) \to \widehat{CH}^{r+q}(X)$ and pushforward operations for morphisms between arithmetic varieties. 
	\par Let $E$ be an algebraic vector bundle over $X$ with a Hermitian metric $h^E$ invariant under the complex conjugation over $X(\comp)$. Then a pair $\overline{E} := (E, h^E)$ is a Hermitian vector bundle over $X$ in the sense of \cite[p. 84]{Soule92}. Let $\widehat{\ch}(\overline{E}) \in \oplus_{k} \widehat{CH}^k(X)_{\mathbb{Q}}$ be the arithmetic Chern character. It satisfies the usual axioms of a Chern character, but it does depend on the choice of the metric. When $h^E$ is replaced by $h^E_{0}$, the difference $\widehat{\ch}(E, h^E) - \widehat{\ch}(E, h^E_{0})$ is given by $(0, \widetilde{{\rm{ch}}} (h^E, h^E_{0}))$, where $\widetilde{{\rm{ch}}} (h^E, h^E_{0})$ is the Bott-Chern secondary characteristic class  (cf. \cite[(1.124)]{BGS1}).
	\par Let $f : X \to B$ be a proper morphism between arithmetic varieties, smooth on generic fiber $X(\mathbb{Q})$. Let $\overline{E}$ be a Hermitian vector bundle over $X$. Grothendieck and Knudsen-Mumford defined an algebraic line bundle $\lambda(E)$ over $B$, see \cite{Knudsen1976}. The fiber at every point $y \in B$ is the alternated tensor product $\lambda(E)_y = \otimes_{q \geq 0} ( \det H^q(f^{-1}(y), E) )^{(-1)^q}$. The line bundle $\lambda(E)$ induces a holomorphic line bundle over the set of complex points $B(\comp)$. Over $X(\comp)$,  the Kähler form $\omega$ induces a Hermitian metric $h^{Tf}$ on the relative tangent bundle $Tf_{\comp}$. This defines a Hermitian vector bundle $ \overline{ Tf_{\comp}} := (Tf_{\comp}, h^{Tf})$. Then one defines Quillen metric on $\lambda(E)$ over $B(\comp)$ as a product of $L^2$ metric and analytic torsion of the fiber, see \cite[Definition 1.12, Theorem 1.15]{BGS3}. Thus, we get a Hermitian line bundle $\lambda(\overline{E})$ over $B(\comp)$. The arithmetic Riemann-Roch theorem of Gillet-Soulé \cite[Theorem VIII.1']{GilSoul92} says 
	\begin{equation}\label{arith_1}
		\widehat{c}_1(\lambda(\overline{E})) = f_*( \widehat{\ch}(\overline{E}) \widehat{\td}(\overline{ Tf_{\comp}})  ) - 
		(0, f_*( \ch(\overline{E}_{\comp}) \td( \overline{ Tf_{\comp}} ) R(  \overline{ Tf_{\comp}} ) ) ),
	\end{equation}
	where $\widehat{\td}(\overline{ Tf_{\comp}}) \in \oplus_{k} \widehat{CH}^k(X)_{\mathbb{Q}}$ is the arithmetic Todd class of $\overline{ Tf_{\comp}}$ and $R$ is the additive genus of Gillet-Soulé defined by the power series $R(z)$ 
	\begin{equation}\label{arith_2}
		  R(z) = \sum_{n \geq 1}^{n \text{ odd}} \Big( 2 \frac{\zeta'(-n)}{\zeta(-n)} + \sum_{j = 1}^n \frac{1}{j} \Big) \zeta(-n) \frac{z^n}{n!}.
	\end{equation}
	Now let's suppose $B = \text{Spec}(\integ)$. Then $\widehat{CH}^1(B) = \real$, and by \cite[Lemma VIII.1.1]{Soule92}, we have 
	\begin{equation}\label{arith_3}
		\textstyle \widehat{c}_1(\lambda(\overline{E})) = \sum_{q = 0}^{n} (-1)^{q} \left( \log \# H^q(X, E)_{tors} - \log \text{Vol}_{L^2}(H^q(X, E))  \right) + 2 \log T(X, \overline{E}),
	\end{equation}
	where $\text{Vol}_{L^2}(H^q(X, E)_{\real})$ is the $L^2$ co-volume of the integer lattice $H(X, E)_{\text{free}}$, which is the free part in cohomology $H^q(X, E) \otimes \real$ and $T(X, \overline{E})$ is the analytic torsion associated with $\overline{E}$ and $\overline{Tf_{\comp}}$. 
	\begin{sloppypar}
	Let's consider the simplest case when $X$ is a projective plane from previous paragraph. For the coordinates $z_0, z_1$ on $\mathbb{C}^2$, we identify the basis of $H^0(\mathbb{CP}^1, \mathcal{O}(p))$ with homogeneous polynomials $x_j = z_0^{j} z_1^{p-j}, j=0,\ldots,p$ of degree $p$.
	By the fact that $\{ x_j \}$ form an orthogonal basis in cohomology with respect to $\norm{\cdot}_{L^2}$, by $\norm{x_j}_{L^2} = j! (p-j)! / (p+1)!$, $H^1(\mathbb{CP}^1, \mathcal{O}(p)) = 0$ for $p \geq 1$ and some calculations with characteristic and secondary characteristic classes, we have
	\begin{multline}\label{eqn_vol_l2}
		 \textstyle \sum_{q = 0}^{1} (-1)^q \log \text{Vol}_{L^2}(H^q(\mathbb{CP}^1, \mathcal{O}(p))) = \sum_{j=0}^{p} \log \norm{x_j}_{L^2}
		   = 2 \log \prod_1^{p} j! - (p+1) \log(p+1)!,
	\end{multline}
	\vspace{-15pt}
	\begin{multline}
		\textstyle f_*( \widehat{\ch}(\overline{E}) \widehat{\td}(Tf_{\comp})  ) - 
		(0, f_*( \ch(\overline{E}_{\comp}) \td( \overline{ Tf_{\comp}} ) R(  \overline{ Tf_{\comp}} ) ) ) \\ 
		= 
		\tfrac{1}{2}(p+1)^2 - 4 \zeta'(-1) -  \tfrac{\log (2 \pi)}{2}p - \tfrac{2 \log (2 \pi)}{3}.	
	\end{multline}
	This goes in line with (\ref{form_zeta_asympt_proj}), (\ref{arith_1}) and (\ref{arith_3}).
	Now, by (\ref{eqn_NT_exp}) and (\ref{eqn_vol_l2}), we have, as $p \to \infty$
	\begin{multline}\label{covol_asy}
		 \sum_{q = 0}^{1} (-1)^q \log \text{Vol}_{L^2}(H^q(\mathbb{CP}^1, \mathcal{O}(p))) = 
		-\frac{1}{2}p^2 - \frac{1}{2}p \log p + \Big( \frac{\log (2 \pi)}{2} - 1 \Big) p  \\
		- \frac{2}{3} \log p + 2 \zeta'(-1) + \frac{1}{2} \log(2 \pi)
 - \frac{13}{12} + O(p^{-1}).
 	\end{multline}
	Thus, the right-hand side of (\ref{covol_asy}) contains $2\zeta'(-1)$	in the constant term.
	It is an interesting question if one could understand the appearance of $\zeta'(-1)$ in the asymptotics of the first summand in the right-hand-side of (\ref{arith_3}) for a general arithmetic variety without using Theorem \ref{dzeta_asympt}.
	\end{sloppypar}
\paragraph{Relation with the result of Klevtsov-Ma-Marinescu-Wiegmann \cite{KlMa}.}\label{sect_KlMA}
	Now let's describe a result from \cite[\S 4]{KlMa}. We denote by $M$ a Riemann surface and by $g_0^{TM}, g_1^{TM}$ Riemann metrics on $M$. Let $L$ be a holomorphic line bundle over $M$ and let $h_0^{L}, h_1^{L}$ be Hermitian metrics such that $L$ is positive with respect to any of $h_0^{L}, h_1^{L}$.
	Let $\matcirc{R^L_{i}} \in \enmr{T^{(1,0)}M}$ be defined as in (\ref{defn_R_L_0}) with respect to $(h_i^{L}, g_i^{TM})$. We denote by $\Delta_{g_i^{TM}}, d v_{i, M}$ the scalar Laplacian and the volume form associated to $g_i^{TM}$.	
	Then \cite[(73)]{KlMa} says
	\begin{equation}
		2 \log T(g_1^{TM}, h_1^{L^p}) - 2 \log T(g_0^{TM}, h_0^{L^p}) = \mathcal{F}_1 - \mathcal{F}_0, 
	\end{equation}
	where 
	\begin{multline}
	 \mathcal{F}_i = - \frac{1}{2} \int_M p \log \Big( \frac{p \matcirc{R^L_{i}}}{2 \pi} \Big) \omega - \frac{1}{3} \int_M \log \Big( \frac{p \matcirc{R^L_{i}}}{2 \pi} \Big) c_1(TM)  \\
	 - \frac{1}{48 \pi} \int_M \log \big( \matcirc{R^L_{i}} \big) \Delta_{g_i^{TM}} \big( \log \matcirc{R^L_{i}} \big) \, d v_{i, M} + O(p^{-1}).
	\end{multline}
	\par The authors observed the equality between the first two terms of the expansion of $\mathcal{F}_1$ and the first two terms of the expansion of $2 \log T(g_1^{TM}, h_1^{L^p})$ (see (\ref{form_B_Vas})), so they conjectured that the third and forth terms will also coincide. We see by Theorem \ref{dzeta_asympt} that the third term of the asymptotic expansion of $2 \log T(g_1^{TM}, h_1^{L^p})$ is $-\tfrac{1}{3} \int_M c_1(TM)$, which coincides with the third term of $\mathcal{F}_1$. The forth term of the asymptotic expansion of $2 \log T(g_1^{TM}, h_1^{L^p})$ is 
	\begin{equation}
		 \textstyle - \frac{1}{24} \rk E \left(24 \zeta'(-1) + 2 \log (2 \pi) + 7 \right) \int_M c_1(TM), 
	\end{equation}
	and in $\mathcal{F}_1$ it is $0$. So the conjecture is valid for the third term, but not for the forth.
	
\section{General asymptotic expansion for orbifolds, Theorem \ref{thm-gen_asympt_orbi}}\label{sect_orbi}
	In this section we prove Theorem \ref{thm-gen_asympt_orbi}. The general framework of  Section \ref{subsec_idea} stays the same. We are still able to do the localization in the calculation of the asymptotic expansion of the analytic torsion. Once we localize the problem, the analysis differs from the manifold's case only in the neighbourhood of singular points, where the problem reduces to the $G-$manifold case and the results of Section \ref{paragr_off_diag} could be applied.
	\par This section is organized as follows. In Section 5.1 we recall the definition of an orbifold and fix some notation. In Section 5.2 we prove some technical lemmas which facilitate further exposition. In Section 5.3 we establish Theorem \ref{thm-gen_asympt_orbi_gen}, which is the full statement of Theorem \ref{thm-gen_asympt_orbi}. We also explain how this theorem implies the main result of Hsiao-Huang \cite{Hsiao16}. 
	
\subsection{Orbifold preliminaries}
	In this section we recall some definitions from orbifolds theory. The content here is taken almost verbatim from the article \cite[\S 1.1]{MaOrbif2005} and the book \cite[\S 5.4]{MaHol}.
	\par We define a category $\mathcal{M}_s$ as follows: the objects of $\mathcal{M}_s$ are the class of pairs $(G, M)$ where $M$ is a connected smooth manifold and $G$ is a finite group acting effectively on $M$. Let $(G, M)$ and $(G', M')$ be two objects, then a morphism $\Phi : (G, M) \to (G', M') $ is a family of open embeddings $\phi : M \to M'$ satisfying:
	\par 1. For each $\phi \in \Phi$, there is an injective group homomorphism $\lambda_{\phi}: G \to G'$ such that $\phi$ is $\lambda_{\phi}$-equivariant.
	\par 2. For $g \in G', \phi \in \Phi,$ we define $g \phi : M \to M'$ by $(g \phi)(x) = g (\phi(x))$ for $x \in M$. If $(g \phi)(M) \cap \phi(M) \neq \emptyset$, then $g \in \lambda_{\phi}(G)$.
	\par 3. For $\phi \in \Phi$, we have $\Phi = \{ g \phi, g \in G' \}.$
	\begin{defn}[Definition of an orbifold]\label{defn_orbi}
		Let $\mm$ be a paracompact Hausdorff space and let $\mathcal{U}$ be a covering of $\mm$ consisting of connected open subsets. We assume $\mathcal{U}$ is \textit{dense}, i.e.
		\par For any $x \in U \cap U', U, U' \in \mathcal{U}$, there is $U'' \in \mathcal{U}$ such that $x \in U'' \subset U \cap U'$. \\
		Then an orbifold structure $\mathcal{V}$ on $\mm$ is the following:
		\par 1. For $U \in \mathcal{U}, \mathcal{V}(U) = ((G_U, \tilde{U}) \to U)$ is a ramified covering, giving an isomorphism $U \simeq \tilde{U}/G_U.$ 
		\par 2. For $U, V \in \mathcal{U}, U \subset V$, there is a morphism $\phi_{VU} : (G_U, \tilde{U}) \to (G_V, \tilde{V})$ that covers the inclusion $U \subset V$.
		\par 3. For $U, V, W \in \mathcal{U}, U \subset V \subset W$, we have $\phi_{WU} = \phi_{WV} \circ \phi_{VU}$.
		If $\, \mathcal{U}'$ is a dense refinement of $\, \mathcal{U}$ we say that the restriction $\mathcal{V}'$ of the orbifold structure $\mathcal{V}$ to $\mathcal{U}'$ is equivalent to $\mathcal{V}$. A pair of $\mm$ and an equivalence class $[\mathcal{V}]$ is called an orbifold.
	\end{defn}
	\begin{rem} This definition corresponds to “an effective orbifold" in the standard terminology.
		\par In Definition \ref{defn_orbi}, we can replace $\mathcal{M}_s$ by a category with manifolds with additional structure (orientation, Hermitian or Riemannian structure) as objects and maps, which preserve this structure, as morphisms. So we can define oriented,  Hermitian or Riemannian orbifolds. 
	\end{rem}
	Let $(\mm, [\mathcal{V}])$ be an orbifold. For each $x \in \mm$, we can choose a small neighbourhood $(G_x, \tilde{U}_x) \to U_x$ such that $x \in \tilde{U}_x$ is a fixed point of $G_x$. The isomorphism class of $G_x$ doesn't depend on the choice of a chart. Let's define $\mm^{sing} = \{ x \in \mm : |G_x| \neq 1 \}$.
	\begin{defn}\label{defn_orbi_v_b}
	 An orbifold vector bundle ${\ee}$ over an orbifold $(\mm, \mathcal{V})$ is defined as follows: ${\ee}$ is an orbifold, for $U \in \mathcal{U}, (G_U^{{\ee}}, \tilde{p}_U : \tilde{{\ee}}_U \to \tilde{U})$ is a $G_U^{{\ee}}$-equivariant vector bundle and $(G_U^{{\ee}}, \tilde{{\ee}}_U)$ is an orbifold structure of ${\ee}$ such that the transition maps in this structure are given by equivariant maps of those vector bundles. Moreover, $(G_U = G_U^{{\ee}} / K_U^{{\ee}}, \tilde{U}), K_U^{{\ee}} = \ker (G_U^{{\ee}} \to {\rm{Diffeo}}(\tilde{U}))$ is an orbifold structure of $\mm$. If $K_U^{{\ee}} = \{1\}$, we call ${\ee}$ a proper orbifold vector bundle.
	\end{defn}
	
	For example, the \textit{orbifold tangent bundle} $T\mm$ of an orbifold $\mm$ is defined by $(G_U, T \tilde{U} \to \tilde{U})$, for $U \in \mathcal{U}$. It is a proper orbifold vector bundle. Let ${\ee} \to \mm$ be an orbifold vector bundle. A section $s : \mm \to {\ee}$ is smooth (or holomorphic if $\mm$ is a complex orbifold), if for each $U \in \mathcal{U}$, $s|_{U}$ is covered by a $G_U^{{\ee}}$-invariant smooth (or holomorphic) section $\tilde{s}_U : \tilde{U} \to \tilde{{\ee}}_U$.
	\par For an oriented orbifold $\mm$ and a form $\alpha$ over $\mm$ (i.e.,  a section of $\Lambda^{\bullet}(T^{*} \mm)$) we define
	\begin{equation}
		\tinyint_\mm \alpha := \tfrac{1}{|G_U|} \tinyint_{\tilde{U}} \tilde{\alpha}_U, \text{ where supp }  \alpha \subset U \in \mathcal{U} 	
	\end{equation}
	We can extend this definition by $\real$-linearity to any differential form with compact support.
	
	\begin{lem}[cf. {\cite[Lemma 5.4.3]{MaHol}}]\label{lem_local_linear}
		We can choose local coordinates $\tilde{U}_x \subset \real^n$ (or $\comp^n$ if orbifold is complex) such that the finite group $G_x$ acts linearly (or $\comp$-linearly) on $\tilde{U}_x$.
	\end{lem}
	
	Let $(1), (h_x^{1}), \cdots, (h_x^{\rho_x})$ be all the conjugacy classes in $G_x$. Let $Z_{G_x}(h_x^{j})$ be the centralizer of $h_x^{j}$ in $G_x$. We also denote by $\tilde{U}_x^{h_x^{j}}$ the fixed point set of $h_x^{j}$ in $\tilde{U}_x$. Then we have a natural bijection 
	$$ \{ (y, (h_y^{j})) | y \in U_x, j=1, \cdots, \rho_y \} \simeq \coprod_{j=1}^{\rho_x} \tilde{U}_x^{h_x^{j}} / Z_{G_x}(h_x^{j}). $$
	\begin{defn}[Strata of an orbifold]\label{defn_strata}
	We can globally define 
	$$ \Sigma \mm = \{ (x, (h_x^{j})) | x \in \mm, G_x \neq 1, j = 1, \cdots, \rho_x \} $$
	and endow $\Sigma \mm$ with a natural orbifold structure defined by 
	$$ \big\{ (Z_{G_x}(h_x^{j}) / K_x^{j}, \tilde{U}_x^{h_x^{j}}) \to \tilde{U}_x^{h_x^{j}} / Z_{G_x}(h_x^{j}) \big\}_{(x, U_x, j)}, $$
	where $K_x^{j}$ is the kernel of the representation $Z_{G_x}(h_x^{j}) \to {\rm{Diff}}(\tilde{U}_x^{h_x^{j}})$ and ${\rm{Diff}}(\tilde{U}_x^{h_x^{j}})$ is the set of diffeomorphisms of $\tilde{U}_x^{h_x^{j}}$.
	\end{defn}
	Till the end of this section we denote by
	\begin{equation}\label{defn_conn_comp_sigm}
		\Sigma \mm^{[j]}, j \in J \quad \text{the connected components of $\Sigma \mm$}, \qquad n_j = \dim_{\comp} \Sigma \mm^{[j]},
	\end{equation}
	\vspace{-20pt}
	\begin{equation}\label{defn_multipl}
		m_j = |K_j| \quad \text{ the multiplicity of } \Sigma \mm^{[j]},
	\end{equation}
	where $K_j$ was defined in Definition \ref{defn_strata}. 
	We have a natural map $\pi : \Sigma \mm \to \mm, (x, (h_x^{j}))  \mapsto x$. Then $\pi|_{\Sigma \mm^{[j]}}$ is an embedding.
	\begin{sloppypar}
	
\subsection{General setup and some auxiliary lemmas}\label{sect_orbi_gen_set}
	Let's fix a compact Hermitian orbifold $(\mm, g^{T \mm}, \Theta)$ of complex dimension $n$. Then its strata $\Sigma \mm$ is naturally a Hermitian orbifold.
	We fix $x \in \Sigma {\mm}^{[j]}$ and we denote by 
	\begin{center}
	$g_j \in G_x$ some element such that $(x, g_j) \in \Sigma {\mm}^{[j]}$.
	\end{center}
	Let's denote by $\tilde{{\nn}}_{j}$ the normal vector bundle to $\tilde{U}_x^{g_j}$ in $\tilde{U}_x$. We introduce the projection $\pi_{(j)}: \tilde{{\nn}}_{j} \to \Sigma {\mm}^{[j]}$. We see that $\tilde{{\nn}}_{j}$ is naturally endowed with the Hermitian metric. We denote by $d v_{\nn}$ its Riemannian volume form. Exponential mapping gives a map $\phi$ from the neighbourhood of the zero section of $\tilde{{\nn}}_{j}$ to the neighbourhood of $\pi(\Sigma {\mm}^{[j]})$ in $\mm$. We define a function $k_j$ in this neighbourhood of $\pi(\Sigma {\mm}^{[j]})$ by 
	\begin{equation}\label{defn_k_j}
		\phi^* d v_{\mm}(x,Z) = k_j(x, Z) (( \phi \pi)^*  d v_{\Sigma \mm^{[j]}}(x)) \wedge dv_{\nn}(x,Z), 
	\end{equation}
	where $d v_{\mm}, dv_{\Sigma \mm}$ are the Riemannian volume forms of $\mm$ and $\Sigma \mm$ respectively.  We extend the function $k_j$ to the whole $\nn_j$ in such a way that all its derivatives are bounded.
	\end{sloppypar}
	\par Let $({\ee}, h^{\ee})$ be a Hermitian proper orbifold vector bundle on $\mm$. By Lemma \ref{lem_local_linear}, we can define the operator $\dbar{}^{\ee}$ locally on each local chart $\tilde{U}$ and patch it globally. As usually, we define the operators $\dbar{}^{\ee *}, \laplcomp^{\ee}$.  By \cite{MaOrbif2005}, the heat operator  $\exp(- t\laplcomp^{\ee})$ has a smooth kernel $\exp(- t\laplcomp^{\ee})(x,y), x,y \in \mm$ with respect to $dv_{\mm}$.
	\par Let $({\lin}, h^{\lin})$ be a holomoprhic Hermitian proper positive orbifold line bundle on $\mm$. We denote by $\theta_j \in 2\pi \rat, j \in J$ the number such that for any $x \in \pi(\Sigma \mm^{[j]})$ 
	\begin{equation}\label{defn_theta_j}
		\text{the action of $g_j \in G_x$ on $\lin_x$ is given by $e^{\imun \theta_j}$.}
	\end{equation}		
	This number is independent of the choice of $x$ and $g_j$. We denote by $\laplcomp_p$ the Laplacian $\laplcomp^{\lin^{p} \otimes \ee}$  and define the analytic torsion $T(g^{T \mm},h^{\lin^{p} \otimes \ee})$ as in Definition \ref{defn-RS_anal_torsion}. We have
	\begin{thm}[{\cite[Theorem 5.4.9]{MaHol}}]\label{thm_orbif_spec_gap}
		There exists $\mu > 0$ such that for any $p \gg 1$, we have
		$$\spec ( \laplcomp_{p}) \subset \{ 0 \} \cup [\mu p, + \infty[, \qquad  \ker (\laplcomp_{p}) \subset \dfor[(0,0)]{M, L^p \otimes E}. $$
	\end{thm}	 
	\begin{sloppypar}
		Locally, over an orbifold chart, we define the function $k$ (see (\ref{defn_k_fun})) and the operators $\widetilde{L_{p,x}}, \widetilde{L_{2,x}^{t}}, \widetilde{L_{4,x}^{t,v}}$ (see (\ref{defn-L_p}), (\ref{defn-L_2^t}) and (\ref{defn_L_4_t})) as we did in the manifolds case. Those objects are $G_x$-invariant. For brevity, we note for $w > 0$
	\end{sloppypar}
	\begin{equation}\label{not_orb_exp}
		e^{-w L_{2, x}^{t}} = \exp(- w  \widetilde{L_{2,x}^{t}}), \qquad e^{-w L_{4,x}^{t,v}} = \exp(- w \widetilde{L_{4,x}^{t,v}}).	
	\end{equation}
	Now we write down some simple corollaries of Section \ref{paragr_off_diag}, which simplify largely the proof of Theorem \ref{thm-gen_asympt_orbi}. Here and after, let $(g_1, g_2) \in G_x \times G_x$ acts on 
	$$ (\xi_1, \xi_2) \in (\Lambda^{\bullet} ( T^{*(0,1)}_{y} \mm ) \otimes {\lin}^p_{y} \otimes {\ee}_y) \otimes (\Lambda^{\bullet} (T^{*(0,1)}_{z} \mm) \otimes {\lin}^p_{z} \otimes {\ee}_z)^{*}, y, z \in \mm, \text{ by} $$
		$$ (g_1, g_2) (\xi_1, \xi_2) = (g_1 \xi_1, g_2 \xi_2) \in (\Lambda^{\bullet} (T^{*(0,1)}_{g_1 y} \mm) \otimes {\lin}^p_{g_1 y} \otimes {\ee}_{g_1 y}) \otimes (\Lambda^{\bullet} (T^{*(0,1)}_{g_2 z} \mm) \otimes {\lin}^p_{g_2 z} \otimes {\ee}_{g_2 z})^{*}. $$
	\begin{lem}\label{lem_aux_0}
		 For any $j \in J, u > 0$ fixed, the function
		 \begin{equation}
		 	\textstyle \int_{Z \in \tilde{{\nn}}_{j,x}} 
			\str{N (g_j,1) e^{-u L_{2, x}^{t}}(g_j^{-1} Z, Z)} (k^{-1} k_j)(x, tZ) \, dv_{{\nn}_j}(Z)
		 \end{equation}
		is differentiable in $(x, t) \in \mm \times [0,1]$ and it's derivatives $ \tfrac{\partial^{a + b}}{\partial x^a \, \partial t^b}|_{t=0} $ vanish for $b$ odd.
	\end{lem}
	\begin{proof}
		First of all, the integral makes sense due to Theorem \ref{thm_L_2_t_off_diag_u_infty} and to the fact that the action of $g_j$ on $\tilde{{\nn}}_{j,x}$ has no fixed points. Due to the $G_x$-invariance of $\widetilde{L_{2,x}^{t}}$ and the fact that $g_j$ acts by isometries, the integral doesn't depend on the choice of $g_j$.
		\par The first part is a consequence of the Lebesgue dominated convergence theorem and Theorem \ref{thm_L_2_t_off_diag}. The vanishing result follows from the Duhamel's formula (\ref{eqn_a_1_volterra}) (cf.  \cite[Theorem 4.17]{MaBerg06}), \cite[Theorem 4.1.7]{MaHol} and the fact that $\exp ( - u \widetilde{L_{2,x}^{0}} )(g_j^{-1} Z, Z)$ is an even function in $Z$ (cf. \cite[Appendix D]{MaHol}).
	\end{proof}	
	
	\begin{lem}\label{lem_aux_2}
	 For any $j \in J$, the function
	 \begin{equation}
		\textstyle \int_{Z \in \tilde{{\nn}}_{j,x}} 
			\str{N (g_j,1)e^{-L_{4,x}^{t,v}}(g_j^{-1} Z, Z)} (k^{-1} k_{j})(x, tv Z) \, dv_{{\nn}_j}(Z)
	 \end{equation}
			is differentiable in $(x, t, v) \in \mm \times [0,1] \times [0, + \infty [$ and it's derivatives $ \tfrac{\partial^{a + b + c}}{\partial x^a \, \partial t^b \, \partial v^c}|_{t=0, v=0} $ vanish whenever $b$ or $c$ is odd.
	\end{lem}
	\begin{proof}
		By Proposition \ref{prop_l_4_off_diag}, the proof is the same as in Lemma \ref{lem_aux_0}.
	\end{proof}	
	We fix $t_0$ as in (\ref{eqn_L_t_spec_gap}).
	\begin{lem}\label{lem_aux_1}
		For any $m \in \nat^*, u_0 > 0$ there are $c, C > 0$ such that for any $u \in ]0, u_0]; j \in J$
		\begin{align}
			&  \Big|
			\int_{Z \in \tilde{{\nn}}_{j,x}} 
			\str{N (g_j,1) e^{-u L_{2, x}^{t}} (g_j^{-1} Z, Z)} (k^{-1} k_{j})(x, tZ) \, dv_{{\nn}_j}(Z) \nonumber
			\\
			& 
			\qquad -u^{-n_j} \int_{Z \in \tilde{{\nn}}_{j,x}} 
			\str{N (g_j,1) e^{-L_{4,x}^{t,v}}(g_j^{-1} Z, Z)} (k^{-1} k_{j} )(x, tvZ) \, dv_{{\nn}_j}(Z)
		 	\Big|_{\mathscr{C}^{m}(\mm \times [0, \, t_0])} \nonumber \\
		 	&  \quad \qquad \qquad \leq C \exp ( - c / u ).
		\end{align}
	\end{lem}
	\begin{proof}
		Let's fix $\epsilon > 0$ small enough. 
		We break up the integral
		\begin{equation}
			 \int_{Z \in \tilde{{\nn}}_{j,x}} 
			\str{N (g_j,1) e^{-u L_{2, x}^{t}}(g_j^{-1} Z, Z)} (k^{-1} k_{j})(x, tZ) 
			\, dv_{{\nn}_j}(Z)  		
		\end{equation}
		into two parts $I_1 = \int_{|Z| \leq \epsilon}$  and $I_2 = \int_{|Z| > \epsilon} $. Similarly, we break the integral 
		\begin{equation}
			 \int_{Z \in \tilde{{\nn}}_{j,x}} 
			\str{N (g_j,1) e^{-L_{4,x}^{t,v}}(g_j^{-1} Z, Z)} (k^{-1} k_{j})(x, tvZ) 
			\, dv_{{\nn}_j}(Z)		
		\end{equation}
		into two parts $J_1 = \int_{|Z| \leq \epsilon/v}' $ and $J_2 = \int_{|Z| > \epsilon/v}'$.
		By Theorem \ref{thm_L_2_t_off_diag} and Proposition \ref{prop_l_4_off_diag}, there are constants $c,C > 0$ such that $| I_2|_{\mathscr{C}^{m'}(\mm)} , |J_2|_{\mathscr{C}^{m'}(\mm)} \leq C \exp(-c/u)$. By Proposition \ref{prop_l_2_l_4}, we get the estimate
		$
			\left| I_1 - J_1 \right|_{\mathscr{C}^{m'}(\mm)} \leq C \exp(-c/u)
		$.
	\end{proof}
	\begin{lem}\label{lem_aux_3}
		For any $u_0>0, m \in \nat$ there exist $c, C > 0$ such that for any $u > u_0, j \in J$
		\begin{equation}
		\textstyle
		\big| 
			\int_{Z \in \tilde{{\nn}}_{j,x}} 
			\str{N (g_j,1) e^{-u L_{2, x}^{t}}(g_j^{-1} Z, Z)} (k^{-1} k_{j})(x, tZ)  \, dv_{{\nn}_j}(Z)
		 	\big|_{\mathscr{C}^m(\mm \times [0, \, t_0])}
		 	 \leq C \exp ( - c u ).
		\end{equation}
	\end{lem}
	\begin{proof}
		It follows from Lebesgue dominated convergence theorem and Theorem \ref{thm_L_2_t_off_diag_u_infty}.
	\end{proof}
	
	\begin{lem}\label{lem_aux_4}
		For any $u_0 > 0; m, k' \in \nat$, there exists $C > 0$ such that for any $u \in ]0, u_0], v = \sqrt{u}, j \in J$, we have
		\begin{multline}
			\Big| 
			\int_{Z \in \tilde{{\nn}}_{j,x}} 
			\str{N (g_j,1) e^{-u L_{2, x}^{t}}(g_j^{-1} Z, Z)} (k^{-1} k_{j})(x, tZ)  \, dv_{{\nn}_j}(Z)  \\
			 -
			\sum_{h = 0}^{k' + n_j} \frac{u^{h}}{(2h)!} \frac{\partial^{2h}}{\partial v^{2h}} 
			\Big( \int_{Z \in \tilde{{\nn}}_{j,x}}
			\str{N (g_j,1) e^{- L_{4,x}^{t,v}}(g_j^{-1} Z, Z)} \\ 
			\cdot (k^{-1} k_{j})(x, tvZ)  \, dv_{{\nn}_j}(Z) 
			\Big)|_{v = 0}
		 	\Big|_{\mathscr{C}^{m}(\mm \times [0, \ t_0])} 
		 	 \leq C u^{k' + 1}.
		\end{multline}
	\end{lem}
	\begin{proof}
		This follows immediately from Lemmas \ref{lem_aux_2}, \ref{lem_aux_1}.
	\end{proof}
	
\subsection{Proof of Theorem \ref{thm-gen_asympt_orbi}}\label{subsect_proof_orbi}
	In this section we prove Theorem \ref{thm-gen_asympt_orbi}. Then we will show that Theorem \ref{thm-gen_asympt_orbi} gives a refinement of the main result of Hsiao-Huang \cite{Hsiao16}. One of the main ingredients here is Lemma \ref{lem_M_sing_kernel}, which localizes the calculation of the asymptotic expansion of the heat kernel near the singular locus. Once this lemma is established, the main strategy of the proof is the same as in the manifold's case from Section \ref{subsec_idea}. The only technical modification will consist in exploiting the results of Section \ref{paragr_off_diag} on off-diagonal expansion of the heat kernel. We use the notation from Section \ref{sect_orbi_gen_set}.
	\begin{sloppypar}
	We begin by giving the definition of the sections $\widetilde{a_{i,u}}$ of the vector bundle $\enmr{\Lambda^{\bullet} (T^{*(0,1)} \mm) \otimes \ee}$, which are the orbifold's counterparts of $a_{i,u}$, defined in Theorem \ref{thm-gen_asympt_exp}.
	If $x \in \mm$ is nonsingular, we define $\widetilde{a_{i,u}}(x)$ by (\ref{eqn_a_i_u_as_derivative}). If $x$ is singular, we define the following local section (see (\ref{not_orb_exp}))
	\begin{equation}\label{defn_a_i_u_sing}
		a'_{k,u}(x) = \frac{1}{(2k)!} \frac{\partial^{2k}}{\partial t^{2k}} e^{-u L_{2, x}^{t} /2}(0,0) |_{t=0}.
	\end{equation}
	It is $G_U$-invariant over an orbifold neighbourhood $\tilde{U}$, so it gives a section of $\enmr{\Lambda^{\bullet} (T^{*(0,1)} \mm) \otimes \ee}$ over $\tilde{U} / G_U$, which we denote by $\widetilde{a_{i,u}}$. 
	\par We prove in Proposition \ref{prop_mellin_orbi} that $\widetilde{a_{i,u}}$ has an expansion of the form (\ref{eqn_mell_1}) as $u \to 0$ (so Notation \ref{not_coefficients} makes sense), and we can perform the Mellin transform for the trace of $\widetilde{a_{i,u}}$. Now let's state Theorem \ref{thm-gen_asympt_orbi} precisely.
	\end{sloppypar}	 
	\begin{thm}\label{thm-gen_asympt_orbi_gen}
		There are $\widetilde{\alpha_i} , \widetilde{\beta_i} \in \real, i \in \nat$ and $\gamma_{j,i}, \kappa_{j,i} \in \real, j \in J, i \in \nat$ such that the asymptotic expansion (\ref{eqn_orbi_asympt}) holds, as $p \to \infty$.
		Moreover,
		\begin{align}
			& \textstyle \widetilde{\alpha_i}= \int_{\mm} \str{N \widetilde{a_i}^{[0]}(x)} \, dv_{\mm}(x), 
			& \textstyle \widetilde{\beta_i} =  -{\rm{M}}_u \big[ \int_{\mm} \str{N \widetilde{a_{i,u}}(x)} \, dv_{\mm}(x) \big]'(0).
		\end{align}
		Also there are functions $c_{j,u,i}, j \in J, u \in ]0, + \infty[, i \in \nat$ on $\Sigma \mm^{[j]}$, given by (\ref{form_c_j_u_i}), such that
		\par 1. For $x \in \mm$ the value $c_{j,u,i}(x)$ depends only on the local geometry of $\mm$ in $x$ and on the action of $g_j \in G_x$ on the normal bundle $\tilde{{\nn}}_{j,x}$,
		\par 2. The equations (\ref{eqn_mell_1}), (\ref{eqn_mell_2}) hold for the functions $u \mapsto c_{j,u,i}(x), \, u > 0$, so we can apply the Mellin transform, and Notation \ref{not_coefficients} makes sense.
		We have the following identities:
		\begin{align}\label{thm_gen_orbi_iden}
			& \textstyle \gamma_{j, i} = \int_{\Sigma \mm^{[j]}} c_{j,i}^{[0]}(x) \, dv_{\Sigma \mm^{[j]}}(x), 
			& \textstyle \kappa_{j, i} =  - {\rm{M}}_u \left[ \int_{\Sigma \mm^{[j]}} c_{j,u,i}(x) \, dv_{\Sigma \mm^{[j]}}(x) \right]'(0).
		\end{align}
		Finally, the identities (\ref{coroll}), (\ref{eqn_gamma_form}) hold and the proportion $\kappa_{j, 0} / { \rm{Vol}} (\Sigma \mm^{[j]})$ depends only on the action of $g_j \in G_x, (x, g_j) \in \Sigma \mm^{[j]}$ on the normal bundle $\tilde{{\nn}}_{j,x}$ of a fixed point $x \in \Sigma \mm^{[j]}$, and for $c_j$ from (\ref{eqn_gamma_form}), we have a precise formula 
		\begin{equation}\label{form_c_j_final}
			c_j = -n \rk{E} \big( \det ({\rm{Id}} - g_j|_{\tilde{{\nn}}_{j}}) \big)^{-1/2}.
		\end{equation}
	\end{thm}
	Now we give a proof of Theorem \ref{thm-gen_asympt_orbi_gen}. 
	Let $\epsilon > 0$ be small enough, we introduce
	\begin{equation}\label{eqn_A_p_u_B_p_u_defn}
		\begin{aligned}
			& A(p, u) = \int_{\mm \setminus B(\mm^{sing}, \epsilon)} \str{N \exp ( - u \laplcomp_p / p )(x,x)} \, dv_{\mm}(x) 
			 \\ 
			& \phantom{A(p, u) = \sum_{j \in J} \tfrac{p^{n_j} e^{\imun \theta_j p}}{m_j} \int_{\Sigma \mm^{[j]}}}+  p^n \int_{B(\mm^{sing}, \epsilon)} \str{N 	
			e^{-u L_{2, x}^{t}/2}(0,0)} \, dv_{\mm}(x).\\ 
			& B(p, u) = \sum_{j \in J} \frac{1}{m_j} p^{n_j} e^{\imun \theta_j p} \int_{\Sigma \mm^{[j]}} \int_{Z \in \tilde{{\nn}}_{j,x}} \str{N 		
			(g_j,1) e^{-u L_{2, x}^{t}/2}(g_j^{-1} Z, Z)}
			\\  
			& \phantom{B(p, u) = \sum_{j \in J} \tfrac{p^{n_j} e^{\imun \theta_j p}}{m_j} \int_{\Sigma \mm^{[j]}}}\, \cdot
			(k^{-1} k_j)(x, tZ) \, dv_{{\nn}_j}(Z) \, dv_{\Sigma \mm^{[j]}}(x).
		\end{aligned}
	\end{equation}
	The following Lemma which explains the difference of the manifold's case and the orbifold's case, and why the results from Section \ref{paragr_off_diag} are necessary for the orbifold's case.
	\begin{lem}\label{lem_M_sing_kernel}
		For $\epsilon > 0$ small enough, there are $c, C > 0$ such that for any $p \in \nat^*, u > 0$:
		\begin{align}\label{eqn_decomp}
			\str{N \exp ( - u \laplcomp_p / p ) } = A(p, u) + B(p, u) + O (  p^{C} \exp ( - c p / u ) ).
		\end{align}
	\end{lem}		
	\begin{proof}
		We suppose $\epsilon$ satisfies (\ref{eqn_inj_rad}) and (\ref{curv_positive}). 
		Let $f : \real \to [0,1] $ be a smooth even function, which satisfies
		\begin{equation} \label{eqn_f_defn}
			f(v) = 
			\begin{cases} 
      			1, & \text{ for } |v| \leq \epsilon/2,  \\
      			0, & \text{ for } |v| \geq \epsilon.
  			\end{cases}
		\end{equation}
		 For $u > 0, a \in \comp$ we denote  holomorphic even functions $F_u, G_u$ on $\comp$ by (cf. \cite[(1.6.13)]{MaHol})
		\begin{equation}
		\begin{aligned}
			& F_u(a) = \int_{- \infty}^{+\infty} e^{iva} \exp(- v^2/ 2) f(\sqrt{u} v) \frac{d v}{\sqrt{2 \pi}}, \\
			& G_u(a) = \int_{- \infty}^{+\infty} e^{iva} \exp(- v^2/ 2) (1- f(\sqrt{u} v)) \frac{d v}{\sqrt{2 \pi}}.
		\end{aligned}
		\end{equation}
		Then we have (cf. \cite[(1.6.14)]{MaHol})
		\begin{equation}\label{eqn_aux_sum_F_G}
			F_u(v D_p) + G_u(v D_p) = \exp (- v^2 D_p^{2} / 2 ). 
		\end{equation} 
		By \cite[Proposition 1.6.4, (5.5.72)]{MaHol}, there are $c, C > 0, k \in \nat$ such that for any $x, x' \in M, p \in \nat^*, u > 0$
		\begin{equation}\label{eqn_aux_G_p_form_1}
			\textstyle \big| G_{ u / p } ( \sqrt{u/p} D_p )(x, x') \big| \leq C p^k \exp(- c p / u).
		\end{equation}
		\par We construct an open cover of $B(\mm^{sing}, \epsilon)$ by a finite number of balls $B_i := B(x_i, \epsilon_i), i \in I$ for $\epsilon_i < 2 \epsilon$ and $x_i \in \mm^{sing}$. 
		We require that $x \in B(\pi( \Sigma \mm^{[j]} ), \epsilon) \cap B_i$ implies $x_i \in \pi( \Sigma \mm^{[j]} )$, for a natural embedding $\pi : \Sigma \mm \to \mm$. 
		We construct a partition of unity $\rho_i$ subordinate to $B_i, i \in I$. 
		We implicitly identify a neighbourhood of $0 \in T_{x_i}\mm$, parametrized by variable $Z$, with a neighbourhood of $x_i$. 	
		\par By the identity $g F_{u / p} ( \sqrt{u / p} D_p ) = F_{u / p} ( \sqrt{u / p} D_p ) g$ and finite propagation speed of solutions of the hyperbolic equations, we have the identity (cf. \cite[(5.4.18)]{MaHol})
		\begin{multline}\label{eqn_orbi_aux_oper_off}
			\textstyle F_{u / p} ( \sqrt{u / p} D_p ) (\exp_{x_i}(Z), \exp_{x_i}(Z)) \\ 
			\textstyle = \sum_{g \in G_{x_i}} (g,1) F_{u / p} ( \sqrt{u / p} \tilde{D}_p ) (\exp_{x_i}(g^{-1} \tilde{Z}), \exp_{x_i}(\tilde{Z})) ,
		\end{multline}
		where $ Z \in T_{x_i} \mm, |Z| \leq \epsilon, \tilde{Z} \in T_{x_i} \tilde{U}_{x_i}$ represents $Z$ in the orbifold chart $\tilde{U}_{x_i}$ of $x_i$.		
		\par In the following series of identities we use Remark \ref{lem_lapl_l_p_off_diag}, (\ref{eqn-L_p_L^t_idty}), (\ref{defn_multipl}), (\ref{eqn_aux_G_p_form_1}), (\ref{eqn_orbi_aux_oper_off}) and the fact that $k$ is $G_x$-invariant.
		\begin{align}
			& \textstyle \int_{B(\mm^{sing}, \epsilon)} \str{N \exp ( - u \laplcomp_p / p )(x,x)} \, dv_{\mm}(x) 
			\\	\nonumber
			& \textstyle \phantom{\int} = 
			\sum_{j \in I} p^n \int_{B(x_j, \epsilon_i)} \rho_j(x) \str{N 	
			e^{-u L_{2, x}^{t}/2}, 0)} \, dv_{\mm}(x) 
			\\			\nonumber
			& \textstyle \phantom{\int =} + \sum_{j \in I} \int_{B(x_j, \epsilon_j)} \rho_j(Z) \str{N 
			\sum_{g \in G_{x_j} \setminus \{1\}}			
			(g,1)\exp ( - u \widetilde{L_{p,\pi_{(j)}(Z)}} / (2p) )(g^{-1}Z,Z)} \, dv_{\mm}(Z) 
			\\ \nonumber
			& \textstyle \phantom{\int =} + O (  p^{C'}   \exp ( - c p / u ) ) \\ \nonumber
			& \textstyle \phantom{\int} =	
			p^n \int_{B(\mm^{sing}, \epsilon)} \str{N 	
			e^{-u L_{2, x}^{t}/2}(0, 0)} \, dv_{\mm}(x) \\ \nonumber
			& \textstyle \phantom{\int =}  + p^n \sum_{j \in J} \frac{1}{m_j} e^{\imun \theta_j p} \int_{\Sigma \mm^{[j]}} \int_{Z \in \tilde{{\nn}}_{j,x}, |Z| \leq \epsilon} \str{N 			
			(g_j,1) e^{-u L_{2, x}^{t}/2}(g_j^{-1}Z/t, Z/t)} 
			\\			 \nonumber
			& \textstyle \phantom{\int = \int p^n \int_{B(\mm^{sing}) } \int_{\Sigma \mm^{[j]}} } \cdot (k^{-1} k_j)(x, tZ) \, dv_{{\nn}_j}(Z) dv_{\Sigma \mm^{[j]}}(x)
			+ O (  p^{C} \exp ( - c p / u ) ),
		\end{align}
		After a change of variables $Z \to Z/t$ and an application of Theorems \ref{thm_L_2_t_off_diag_u_infty}, \ref{thm_L_2_t_off_diag} we conclude.
	\end{proof}
	The terms $A(p,u), B(p,u)$ appear in each proof till the end of this section. Due to the finite propagation speed of solutions of the hyperbolic equations, the analysis of $A(p,u)$ is always the same as in Section \ref{subsec_idea}. The main contribution here is the analysis of $B(p,u)$. 
	The following proposition is an orbifold's version of Theorem \ref{thm-gen_asympt_exp}.
	\begin{prop}\label{prop_str_general_orbi}
		For any $k \in \nat$ and $u>0$ fixed, we have as $p \to + \infty$
		\begin{multline}
			 \str{N \exp ( - u \laplcomp_p / p ) } =  
			\sum_{i=0}^{k} p^{n - i} \int_{\mm} \str{N \widetilde{a_{i,u}}(x)} \, dv_{\mm}(x) 
			\\ 
			+ \sum_{i=0}^{k + n_j - n} \sum_{j \in J} \frac{p^{n_j - i}}{m_j} e^{\imun \theta_j p}  \int_{\Sigma \mm^{[j]}} c_{j,u,i}(x) \, dv_{\Sigma \mm^{[j]}}(x) + o(p^{n-k}),
		\end{multline}
		where
		\begin{equation}\label{form_c_j_u_i}
			 c_{j,u,i}(x) = \frac{1}{(2i)!} \int_{\tilde{{\nn}}_{j,x}} {\rm{Tr}_s} \Big[
			N (g_j, 1) \frac{\partial^{2i}}{\partial t^{2i}} \Big( e^{-u L_{2, x}^{t}/2}(g_j^{-1}Z, Z) (k^{-1} k_j)(x, tZ) \Big)|_{t = 0} 
			\Big] dZ,
		\end{equation}
		and the term $o(p^{n-k})$ is uniform as $u$ varies in compact subsets of $]0, + \infty[$.		 
	\end{prop}
	\begin{proof}
		By Lemma \ref{lem-lapl_L_p-kernels}, Proposition \ref{prop_cinf_prol} and (\ref{eqn_A_p_u_B_p_u_defn}), we get
		\begin{equation}\label{eqn_aux_orbi_general}
			\textstyle A(p,u)
			= \sum_{i=0}^{k} p^{n - i} \int_{\mm} \str{N \widetilde{a_{i,u}}(x)} \, dv_{\mm}(x) 
			 + o(p^{n-k}).
		\end{equation}
		By Lemma \ref{lem_aux_3}, (\ref{eqn_A_p_u_B_p_u_defn}) and (\ref{form_c_j_u_i}), we get
		\begin{equation}\label{eqn_aux_orbi_general_1}
			\textstyle B(p,u) = \sum_{i=0}^{k + n_j - n} \sum_{j \in J} \frac{1}{m_j} p^{n_j - i} e^{\imun \theta_j p} \int_{\Sigma \mm^{[j]}} c_{j,u,i}(x) \, dv_{\Sigma \mm^{[j]}}(x) + o(p^{n-k}).
		\end{equation}
		Now, Lemma \ref{lem_M_sing_kernel}, (\ref{eqn_aux_orbi_general}) and (\ref{eqn_aux_orbi_general_1}) imply the proposition. 
	\end{proof}
	The next proposition is an analogue of Proposition \ref{exp_a_i_u}. It implies that we can do the Mellin transform in $u$ for $\widetilde{a_{i,u}}(x), c_{j,u,i}(x)$; thus, the statement of Theorem \ref{thm-gen_asympt_orbi_gen} makes sense.
	\begin{prop}\label{prop_mellin_orbi}
		There are smooth sections $\widetilde{a_i}^{[l]}(x), i \in \nat, l \in \integ, l \geq -n$ of the vector bundle $\enmr{\Lambda^{\bullet} ( T^{*(0,1)}\mm ) \otimes \ee}$ such that the following asymptotic expansion holds for any $ k \in \nat$:
		\begin{equation}
		\textstyle \widetilde{a_{i,u}}(x) = \sum_{l = -n}^k \widetilde{a_i}^{[l]}(x) u^l + o(u^k), \quad \text{ as } \quad u \to 0.
		\end{equation}
		 Moreover, there are $c_i, d_i > 0$ such that we have the following estimation
		\begin{equation}
			 | N \widetilde{a_{i,u}}(x) | \leq c_i \exp(-d_i u).
		\end{equation}
		Similarly for $i \in \nat, j \in J, h \in \integ, h \geq -n_j$ there are functions $c^{[h]}_{ji}(x)$ on $\Sigma \mm^{[j]}$ such that for any $ k \in \nat$, we have the following asymptotic expansion, as $u \to 0$:
		\begin{equation}\label{eqn_c_jui_near_0}
		\textstyle		c_{j,u,i}(x) = \sum_{h = - n_j}^{k} c^{[h]}_{ji}(x) u^{h} + o(u^k). 
		\end{equation}
		Moreover, there are $c_{i}, d_{i} > 0$ such that we have the following estimation
		\begin{equation}\label{eqn_c_jui_infty}
			| c_{j,u,i}(x) | \leq c_{i} \exp(-d_{i} u).
		\end{equation}
	\end{prop}
	\begin{proof}
		The statements about $\widetilde{a_{i,u}}(x)$ are proved in the same way as Proposition \ref{exp_a_i_u}.
		Estimation (\ref{eqn_c_jui_near_0}) follows from Lemma \ref{lem_aux_4} and (\ref{form_c_j_u_i}). 
		Moreover, it proves that
		\begin{multline}\label{form_c_j_p_i}
			 c^{[h]}_{jk}(x) = \frac{1}{(2k)! (2(h+n_j))!} 
			\int_{Z \in \tilde{\nn}_{j,x}}  {\rm{Tr}_s} \Big[ N (g_j, 1) \\ \cdot
			\frac{\partial^{2(k+h + n_j)}}{\partial t^{2k} \partial v^{2(h + n_j)}}  
			 \Big( e^{-L_{4,x}^{t,v}/2}(g_j^{-1} Z, Z) (k^{-1} k_j)(x, tvZ) \Big) |_{t=0, v=0}	 \Big] \, dZ.
		\end{multline}
		Estimation (\ref{eqn_c_jui_infty}) follows from Lemma \ref{lem_aux_3} and (\ref{form_c_j_u_i}). 
	\end{proof}
	We also have the following version of Proposition \ref{prop_b_p_i_exp} and (\ref{defn-b_pj}).
	\begin{prop}\label{prop_b_p_i_exp_orbi}
		For $ p \in \nat^*$ there are $\widetilde{b_{p,i}} \in \comp, i \in \integ,  i \geq -n$ and $\widetilde{b_{j,p,i}} \in \comp, j \in J, i \in \integ, i \geq -n_j$ such that for any $ k \in \nat$, we have the following asymptotic expansion, as $u \to 0$:
		\begin{equation}\label{eqn_b_i_defn_exp_orbi}
			 \str{N \exp ( - u \laplcomp_p / p )} = p^n \sum_{i = -n}^{k} \widetilde{b_{p,i}} u^i + \sum_{j \in J} \frac{p^{n_j}}{m_j}  e^{\imun \theta_j p} \sum_{i = -n_j}^{k} \widetilde{b_{j,p,i}} u^{i} + o(u^k),
		\end{equation}
		We also have the following expansions as $p \to + \infty$
		\begin{align}\label{exp_b_p_i}
			& \textstyle \widetilde{b_{p,i}} = \sum_{l=0}^{k} \widetilde{b_{i}}^{[l]} p^{-l} + o(p^{-k}),  
			&&\text{with}
			&&&\textstyle \widetilde{b_{i}}^{[l]} = \int_{\mm} \str{N \tilde{a_{l}}^{[i]}(x)} \, dv_{\mm}(x),
			\\\label{form_b_tilde_a}
			& \textstyle \textstyle \widetilde{b_{j,p,i}} = \sum_{h=0}^{k} \widetilde{b_{ji}}^{[h]} p^{-h} + o(p^{-k}), 
			&&\text{with}
			&&&\textstyle \widetilde{b_{ji}}^{[h]} = \int_{\Sigma \mm^{[j]}} c_{jh}^{[i]}(x) \, dv_{\Sigma \mm^{[j]}}(x). 
		\end{align}
		Moreover, for $x \in \Sigma \mm^{[j]}$, $\widetilde{b_{j,p,i}}(x)$ depends only on the geometry of $\mm^{[j]}$, ${\nn}_{j}$ at $x$ and on the action of $g_j$ on $\tilde{{\nn}}_{j,x}$.
	\end{prop}
	\begin{proof}
		Similarly to Proposition \ref{prop_b_p_i_exp}, by (\ref{eqn_A_p_u_B_p_u_defn}), we get
		\begin{equation}\label{eqn_orbi_suppl_1}
			\textstyle A(p,u)  = p^n \sum_{i = -n}^{k} \widetilde{b_{p,i}} u^i + o(u^k).
		\end{equation}
		This proves (\ref{exp_b_p_i}).
		Now let's denote for $j \in J, i \in \nat, i \geq -n_j$
		\begin{multline}\label{form_b_j_p_i}
			 \widetilde{b_{j,p,i}} =  \frac{1}{(2(i + n_j))!} \int_{\Sigma \mm^{[j]}} \int_{Z \in \tilde{{\nn}}_{j,x}} \frac{\partial^{2(i + n_j)}}{\partial v^{2(i + n_j)}} \Big(	
			\str{N 
			(g_j,1) 
			e^{-L_{4,x}^{t,v}/2}(g_j^{-1} Z, Z)} 
			\\  
			\vphantom{\str{N e^{-L_{4,x}^{t,v}/2}}}
			\cdot (k^{-1}k_j)(x, tvZ)
			\Big)|_{v=0}
			\, dv_{{\nn}_j}(Z) \, dv_{\Sigma \mm^{[j]}}(x).
		\end{multline}		
			By Lemma \ref{lem_aux_4} and (\ref{eqn_A_p_u_B_p_u_defn}), we get
			\begin{equation}\label{eqn_orbi_suppl_3}
				\textstyle B(p, u) = \sum_{j \in J} \frac{1}{m_j}p^{n_j} e^{\imun \theta_j p} \sum_{i = -n_j}^{k} \widetilde{b_{j,p,i}} u^{i} + o(u^k).
			\end{equation}
			By Lemma \ref{lem_M_sing_kernel}, (\ref{eqn_orbi_suppl_1}) and (\ref{eqn_orbi_suppl_3}), we get (\ref{eqn_b_i_defn_exp_orbi}). Now (\ref{form_b_tilde_a}) follows from Lemma \ref{lem_aux_2} and (\ref{form_c_j_p_i}).
	\end{proof}
	Now we prove the orbifold's analogue of Proposition \ref{prop-streng_zero}.
	\begin{thm}\label{thm_str_zero_orbi}
		For any $k \in \nat, u_0 > 0$ there exists $C > 0$ such that for $u \in ]0, u_0], p \in \nat^*$, we get
		\begin{align}\label{thm_str_zero_orbi_eqn}
			p^k 
			\bigg| 
					& \str{N \exp ( - u \laplcomp_p / p )} - p^n \sum_{i = -n}^{0} \widetilde{b_{p,i}} u^i - \sum_{j \in J} \frac{p^{n_j} }{m_j} e^{\imun \theta_j p} \sum_{i = -n_j}^{0} \widetilde{b_{j,p,i}} u^{i} 
				\\		 
				 &
				 - 
				 \sum_{h = 0}^{n+k}p^{n-h}
				 \Big( 
					\int_{\mm} \str{N \widetilde{a_{h,u}}(x) } \, dv_{\mm}(x) - \sum_{i = -n}^{0} \widetilde{b_{i}}^{[h]} u^i
				\Big)
				\nonumber
				\\				
				&
				-
				\sum_{h = 0}^{k + n_j}
				\sum_{j \in J}
				\frac{p^{n_j-h}}{m_j}e^{\imun \theta_j p}
				\Big( 
					\int_{\Sigma \mm^{[j]}} c_{j, u, h}(x) \, dv_{\Sigma \mm^{[j]}}(x) - \sum_{i = -n_j}^{0} \widetilde{b_{ji}}^{[h]} u^{i} 
				\Big)
			\bigg| \leq C u. \nonumber
		\end{align}
	\end{thm}
	\begin{proof}
		We apply (\ref{eqn_A_p_u_B_p_u_defn}) and the same techniques as in Proposition \ref{prop-streng_zero} to get that there exists $C > 0$ such that for any $u \in ]0, u_0], p \in \nat^*,$
		\begin{multline}\label{eqn_aux_zero_orb_1}
			\textstyle
			p^k 
			\Big|
				\big(
					A(p,u) - p^n \sum_{i = -n}^{0} \widetilde{b_{p,i}} u^i
				\big)
				\\		
				\textstyle		 
				 - 
				 \sum_{h = 0}^{n+k}p^{n-h}
				 \big(
					\int_{\mm} \str{N \widetilde{a_{h,u}}(x) } \, dv_{\mm}(x) -  \sum_{i = -n}^{0} \widetilde{b_{i}}^{[h]} u^i
				\big)
			\Big|
			\leq C u.
		\end{multline}
		By Lemma \ref{lem_aux_4}, (\ref{eqn_A_p_u_B_p_u_defn}), (\ref{form_c_j_u_i}), (\ref{form_b_tilde_a}) and  (\ref{form_b_j_p_i}) we have
		\begin{multline}\label{eqn_aux_zero_orb_2}
			\textstyle			
			p^k
			\Big|
				\Big(
					B(p,u)  - \sum_{j \in J} \frac{1}{m_j} p^{n_j} e^{\imun \theta_j p} \sum_{i = -n_j}^{0} \widetilde{b_{j,p,i}} u^{i} 
				\Big) 
				-
				\sum_{h = 0}^{k + n_j}
				\sum_{j \in J}
				\frac{1}{m_j} p^{n_j-h} 
				\\
				\textstyle
				\cdot  e^{\imun \theta_j p} 
				\Big(
					\int_{\Sigma \mm^{[j]}} c_{j, u, h}(x) \, dv_{\Sigma \mm^{[j]}}(x) - \sum_{i = -n_j}^{0} \widetilde{b_{ji}}^{[h]} u^{i} 
				\Big)
				\Big|
			\leq C u.
		\end{multline}
		Now, by Lemma \ref{lem_M_sing_kernel}, (\ref{eqn_aux_zero_orb_1}) and (\ref{eqn_aux_zero_orb_2}), we get (\ref{thm_str_zero_orbi_eqn}).
	\end{proof}
	Now we prove the orbifold's analogue of Proposition \ref{prop-streng_infty}.
	\begin{thm}\label{thm_infty_orbi}
		For any $k \in \nat, u_0 > 0$ there exists $c, C > 0$ such that for $u > u_0, p \in \nat^*$
		\begin{multline}
			 p^k \bigg| \str{N \exp ( - u \laplcomp_p / p ) } -  
			\sum_{i=0}^{n + k} p^{n - i} \int_{\mm} \str{N \widetilde{a_{i,u}}(x)} \, dv_{\mm}(x) 
			\\
			 - \sum_{i=0}^{k + n_j} \sum_{j \in J} \frac{1}{m_j}p^{n_j - i} e^{\imun \theta_j p}  \int_{\Sigma \mm^{[j]}} c_{j,u,i}(x) \, dv_{\Sigma \mm^{[j]}}(x) \bigg| \leq C \exp(-cu).
		\end{multline}
	\end{thm}
	\begin{proof}
		We have to distinguish two cases:\\
		1. $ u > \sqrt{p}$.
				Similarly to Proposition \ref{prop-streng_infty}, we get by Theorem \ref{thm_orbif_spec_gap} and Proposition \ref{prop_str_general_orbi}. 
				\begin{equation}\label{est_orbi_infty_1}
					\str{N \exp ( - u \laplcomp_p / p ) } \leq \exp(- c u), \qquad
				 \widetilde{a_{i,u}}(x) \leq \exp(-c u). 
				\end{equation}
				We conclude by (\ref{eqn_c_jui_infty}), (\ref{est_orbi_infty_1})  and inequality $\exp(-c u) \leq \exp(-c \sqrt{p}/2 ) \exp(-c u /2)$.\\
		2. $ u \leq \sqrt{p}$.
				Similarly to Proposition \ref{prop-streng_infty}, we get
				\begin{equation}\label{eqn_aux_infty_orb_1}
					\textstyle p^k \Big| A(p,u) -  
					\sum_{i=0}^{n + k} p^{n - i} \int_{\mm} \str{N \widetilde{a_{i,u}}(x)} \, dv_{\mm}(x) 
					 \Big| \leq C \exp(-cu).
				\end{equation}								
				By Lemma \ref{lem_aux_3}, (\ref{eqn_A_p_u_B_p_u_defn}) and (\ref{form_c_j_u_i}), we have
				\begin{equation}\label{eqn_aux_infty_orb_2}
					\textstyle p^k \Big| B(p,u)  
					 -  \sum_{i=0}^{k + n_j} \sum_{j \in J} \frac{1 }{m_j} p^{n_j - i} e^{\imun \theta_j p} \int_{\Sigma \mm^{[j]}} c_{j,u,i}(x) \, dv_{\Sigma \mm^{[j]}}(x) \Big| \leq C \exp(-cu).
				\end{equation}
				We conclude by Lemma \ref{lem_M_sing_kernel}, (\ref{eqn_aux_infty_orb_1}), (\ref{eqn_aux_infty_orb_2}) and inequality 
				$e^{-cp/u} \leq e^{-c\sqrt{p}/2} e^{-cu/2}$.
	\end{proof}
	Now, we can repeat the argument of Theorem \ref{thm-gen_asympt_zet_gen} to get (\ref{eqn_orbi_asympt}). The identities (\ref{coroll}) follow from (\ref{form_B_Vas}) and (\ref{defn_a_i_u_sing}). The identities (\ref{thm_gen_orbi_iden}) are proved in the same way as (\ref{thm-a_i-b_i_formula}). 
	\par \textbf{Now let's explain why the constants $\widetilde{\alpha}_i$, $\gamma_{j,i}$ do not depend on $g^{T \mm}$ ,$h^{\lin^{p}}$, $h^{\ee}$.} In \cite[Theorem 0.1]{MaOrbif2005}, Ma proved an analogue of the anomaly formula for orbifolds. Due to this formula, we have an analogous formula to (\ref{tors_anom}). By \cite[Theorems 1,2]{DLM12} the asymptotic expansion of the Bergman kernel has only terms of the form $p^i$, $p^i e^{\imun \theta_j p}$. Similarly to Proposition \ref{prop_b_p_i_exp_orbi}, we conclude that the orbifolds analogue of the term $M_{0,c}^{p}$ has only terms of the form $p^i$, $p^i e^{\imun \theta_j p}$ in its asymptotic expansion. Thus, under the change of the metric, only the coefficients of the terms $p^i$, $p^i e^{\imun \theta_j p}$ change, so the constants $\widetilde{\alpha}_i$, $\gamma_{j,i}$ do not depend on $g^{T \mm}$ ,$h^{\lin^{p}}$, $h^{\ee}$
	\par \textbf{Now let's prove (\ref{eqn_gamma_form}).} To simplify our calculation, we work under the assumption $\Theta = \omega$. We have the following formula \cite[Appendix E 2.2]{MaHol}
	\begin{equation}
	 	 (g_j, 1) e^{-u L_{2, x}^{0}}(g_j^{-1}Z, Z) = e^{-4 \pi u N} C_{2u}  
		\exp \Big( -\frac{\pi}{\tanh(2 \pi u)} \norm{Z}^2 + \frac{\pi}{\sinh(2 \pi u)} \scal{g_j^{-1} Z}{Z} \Big).
	\end{equation}
	By (\ref{lem_trs_comp}), we get
	\begin{multline}\label{eqn_str_gamma_calc}
		 \str{N (g_j, 1) e^{-u L_{2, x}^{0}}(g_j^{-1}Z, Z)} = -\rk{E} n e^{- 4 \pi u} \frac{1}{1 - e^{-4\pi u}}  \\ 
		\cdot \exp \Big( -\frac{\pi}{\tanh(2 \pi u)} \norm{Z}^2 + \frac{\pi}{\sinh(2 \pi u)} \scal{g_j^{-1} Z}{Z} \Big).
	\end{multline}
	We denote by $\phi_{j, k}, k = 1, \ldots, n - n_j$ the angles of the rotation of the map $g_j|_{\widetilde{{\nn}}_{j,x}}$. From (\ref{form_c_j_u_i}), (\ref{eqn_str_gamma_calc}), we see 
	\begin{equation}\label{form_c_j_0}
		 c_{j,u,0}(x) =  - \frac{\rk{E} n e^{- \pi u} \sinh(\pi u)^{n - n_j-1}}{2 \prod_{k=1}^{n - n_j}  \sqrt{\cosh(\pi u)^2 - 2 \cos(\phi_{j,k}) \cosh(\pi u) + 1}}.
	\end{equation}
	Since $\phi_{j,k} \notin  2 \pi \integ$ for $k \in 1, \ldots, m_j$, the function $c_{j,u,0}(x)$ is continuous at $u =0$. We get  (\ref{eqn_gamma_form}), (\ref{form_c_j_final}) from (\ref{thm_gen_orbi_iden}) and (\ref{form_c_j_0}).
	
\paragraph{Relation between Theorem \ref{thm-gen_asympt_orbi} and the result of Hsiao-Huang in \cite{Hsiao16}.}\label{sect_Hsiao}
	In a recent article \cite{Hsiao16}, Hsiao-Huang  considered a compact connected strongly pseudoconvex CR manifold $X$ with a $S^1$-transversal locally-free CR action. They considered a \textit{rigid} CR vector bundle $E$ over $X$ (being \textit{rigid} is equivalent to being a pull-back of the holomorphic orbifold vector bundle $\ee$ over the quotient $\mm = X / S^1$, see \cite[Definition 2.4]{Hsiao16}). They decomposed the space $\dfor[(0,\bullet)]{X,E}$ into Fourier components 
	\begin{equation}
		\Omega^{(0,\bullet)}_p(X,E) = \{ u \in \dfor[(0, \bullet)]{X, E} | (e^{i \theta})^* u = e^{i p \theta} u, \text{ for all } \theta \in [0, 2 \pi[ \}, \quad p \in \integ.
	\end{equation}
	Now, the restriction $\dbar_{b, p}$ of the tangential Cauchy-Riemann operator $\dbar_{b}$ endows $\Omega^{(0,\bullet)}_p(X,E)$ with a structure of a differential complex. We endow $X$ with $S^1$-invariant Riemannian metric $g^{TX}$, compatible with CR-structure and $S^1$-action, i.e. $g^{TX}$ is the orthogonal sum of a pull-back of $J$-invariant metric over the complex orbifold $(X / S^1, J)$ and the trivial metric induced on the $S^1$-directions, see \cite[p.4, last paragraph]{Hsiao16}. We endow $E$ with $S^1$-invariant Hermitian metric $h^E$ satisfying \textit{rigidity assumption}, i.e. it is a pull-back of a Hermitian metric on $X / S^1$, see \cite[Definition 2.5]{Hsiao16}. Then we can endow $\Omega^{(0,\bullet)}_p(X,E)$ with the Hermitian metric induced by the $L^2$-scalar product. We denote by $\laplcomp_{b, p}$ the \textit{Kohn Laplacian}, defined by $\laplcomp_{b, p} = \dbar_{b, p} \dbar_{b, p}^{\,*} + \dbar_{b, p}^{\,*} \dbar_{b, p}.$ We associate the \textit{analytic torsion} $\widetilde{T}_p(g^{TX},h^{E})$ to $\laplcomp_{b, p}$, as in the case of Kodaira Laplacian (see Definition \ref{defn-RS_anal_torsion}, (\ref{defn_zeta_fun})). The main result in the article \cite[Theorem 1.1]{Hsiao16} is the calculation of the first term of the asymptotic expansion of $\log \widetilde{T}_p(g^{TX},h^{E})$, as $p \to + \infty$.
	\par Let $X$ be a circle bundle associated to the dual of a positive line bundle $L$ over a compact Hermitian complex manifold $M$.
	Let $\ee$ be a vector bundle over a quotient manifold $M = X / S^1$ and $E$ is a vector bundle $\pi^* \ee$ for $\pi : X \to M$.
	The authors constructed a chain isometry \cite[p.2, operator $A_m$]{Hsiao16} between differential complexes $ ( \Omega^{(0,\bullet)}_p(X, E), \dbar_{p,b} )$ and $( \dfor[(0, \bullet)]{M, L^p \otimes \ee}, \dbar{}^{L^p \otimes \ee})$.
	Thus,  $\spec \, \laplcomp_{b, p} = \spec \, \laplcomp^{L^p \otimes \ee}$ and $\widetilde{T}_p(g^{TX},h^{E}) = T(g^{TM}, h^{L^p \otimes \ee})$. By \cite[(1.4), (1.5)]{Hsiao16} in this case their result is equivalent to the original result of Bismut and Vasserot \cite{BVas}.
	\par Now, in general it has been proven by Ornea and Verbitsky \cite[Theorems 1.11, 5.1]{VerbitSasak07} that a compact connected strongly pseudoconvex CR manifold $X$ with a $S^1$-transversal locally-free CR action is a circle bundle associated to the dual of a positive line bundle $\lin$ over a compact Hermitian orbifold $\mm = X / S^1$. Similarly to the case when $\mm$ is a manifold, there is a chain isometry between differential complexes $( \Omega^{(0,\bullet)}_p(X, E), \dbar_{p,b} )$ and $( \dfor[(0, \bullet)]{\mm, \lin^p \otimes \ee}, \dbar{}^{\lin^p \otimes \ee})$.
	Then Theorem \ref{thm-gen_asympt_orbi} gives the asymptotic expansion of $\log T(g^{T \mm}, h^{\lin^p \otimes \ee}), p \to + \infty$. After reformulating this in terms of geometric objects on $X$, as it was done in the case where $\mm$ is a manifold in \cite[(1.4), (1.5)]{Hsiao16}, Theorem \ref{thm-gen_asympt_orbi} implies the main theorem of the article \cite[Theorem 1.1]{Hsiao16}. We also point out that the fractional powers on $p$ in fact do not indeed appear in the asymptotic expansion of $\log \widetilde{T}_p(g^{TX},h^{E})$.

\begin{spacing}{0.7}

		\bibliographystyle{abbrv}

\end{spacing}

\Addresses

\end{document}